\RequirePackage{fix-cm}
\documentclass[twocolumn]{svjour3}          
\smartqed  
\usepackage{graphicx}
\usepackage{amssymb}
\usepackage{array}

\usepackage{amsmath}
\usepackage{stmaryrd}  

\newcommand{\off}[1]{}

\def\IR{\relax{\rm I\kern-.18em R}}
\def\p{\partial}
\def\<{\langle}
\def\>{\rangle}

\def\sgn{\textrm{sgn}}
\DeclareMathOperator{\Div}{div}

\graphicspath{{Eps1/}}

\begin{document}

\title{Semi-Inner-Products for Convex Functionals and Their Use in Image Decomposition}


\author{Guy Gilboa
}


\institute{ Department of Electrical Engineering, \\
            Technion -– Israel Institute of Technology,\\
            Haifa 32000, Israel
            \email{guy.gilboa@ee.technion.ac.il}           
}

\date{Received: date / Accepted: date}


\maketitle
\bibliographystyle{plain}

\begin{abstract}
Semi-inner-products in the sense of Lumer are extended to convex functionals.
This yields a Hilbert-space like structure to convex functionals in Banach spaces.
In particular, a general expression for semi-inner-products with respect to one homogeneous functionals is given. Thus one can use the new operator for the analysis of total variation and higher order functionals like total-generalized-variation (TGV).
Having a semi-inner-product, an angle between functions can be defined in a straightforward manner.
It is shown that in the one homogeneous case the Bregman distance can be expressed in terms of this newly defined angle.
In addition, properties of the semi-inner-product of nonlinear eigenfunctions induced by the functional are derived.
We use this construction to state a sufficient condition for a perfect decomposition of two signals and suggest numerical measures which indicate when those conditions are approximately met.
\keywords{Semi-inner-product \and Total variation \and Nonlinear eigenfunctions \and Image decomposition}
\end{abstract}


\section{Introduction}
Formulating image-processing and computer-vision tasks as variational problems, has been
used extensively, with great succuss for denoising, segmentation, optical flow, stereo matching, 3D reconstruction and more
\cite{ak_book02,Intro_TV_Chambolle2010,TV_Zoo_Burger_Osher2013}.
In those cases regularizing functionals are used to avoid non-physical solutions and to overcome problems related with noisy measurements.
For images, depth and optical-flow maps, and many other modalities - the signals have inherent discontinuities.
Therefore, an appropriate mathematical modeling should account for that.
One-homogeneous functionals, specifically based on the $L^1$ norm, can cope well with discontinuities.
The most classical one is the total variation (TV) functional as first introduced for image processing in \cite{rof92}
and, in recent years, the proposition of total-generalized-variation (TGV) \cite{bredies_tgv_2010}
which has increased the applicability of such regularizers from essentially piecewise constant to piecewise smooth solutions.


Recently, there is an emerging branch of studies trying to use functionals in alternative ways, broadening their analytical scope and usability \cite{Gilboa_spectv_SIAM_2014,Benning_Burger_2013,spec_one_homog15}.
In this context solutions of nonlinear eigenvalue problems induced by the regularizer are assumed as the fundamental structuring elements.
A nonlinear spectral theory is developed, where operations such as nonlinear low-pass and high-pass filters can be performed.

In this paper we introduce an additional necessary ingredient in nonlinear spectral analysis of functionals, which is a weaker form of
the inner-product to Banach spaces. It is referred to as a \emph{semi-inner-product} and was first introduced by Lumer in \cite{Lumer61}.
We define the properties of a semi-inner-product for functionals and present the formulation for the one-homogeneous case.
We then introduce a notion of semi-inner-products of degree $q$, where for $q=1/2$ this definition provides a useful construct.
Properties of semi-inner-products in the case of nonlinear eigenfunctions are discussed, where things simplify considerably.
Finally, we connect these new notions to the problem of image decomposition, see e.g. \cite{Meyer[1],Luminita[1],Aujol[3],agco06,gsz_var06,Tai_multilayer_graph_cut_seg_PR2013}.
A necessary condition for perfect decomposition is stated and soft indicators of how well two signals can be decomposed using a regularizer and
nonlinear spectral filtering are formulated.

\subsection{Main contributions}
The main contributions of the paper are:
\begin{enumerate}
\item Defining the properties of semi-inner-products for general convex functionals
from which angles and orthogonality measures, with respect to a functional, can be derived.
\item Proposing a semi-inner-product formulation for the case of one-homogeneous functionals,
$$ [u,v]_J:= \<u,p(v)\>J(v),  \,\, p(v) \in \p J(v).$$
\item Extending semi inner products to be of degree $q$ and showing the applicability for $q=1/2$.
\item Showing that in the case of $J$ being one-homogeneous the Bregman distance \cite{bre67} can be related to the angle
between the functions $u$ and $v$ by
$$ D_J(u,v) = J(u)\left(1 - \cos(\textrm{angle}(u,v)) \right).$$
\item Connecting the semi-inner-product to image decomposition through the recently proposed variational spectral filtering  approach \cite{Gilboa_spectv_SIAM_2014} and presenting
a sufficient condition for perfect decomposition of functions admitting the nonlinear eigenvalue problem \eqref{eq:eigenfunction}, (in Th. \ref{th:decomp}).
\item Proposing two soft measures to estimate when a good decomposition is expected and validating these through numerical experiments.
\end{enumerate}


\section{Preliminaries}
We will now summarize four mathematical concepts and notions which are at the basis of this manuscript:
\begin{enumerate}
\item The semi-inner-product of Lumer.
\item Convex one-homogeneous functionals and their unique properties.
\item Functions admitting a nonlinear eigenvalue problem induced by a convex regularizer.
\item A recent direction suggested in \cite{Gilboa_spectv_SIAM_2014} of analyzing and processing regularization problems using a nonlinear spectral approach.
\end{enumerate}
We will see at the last section how all these components are brought together in the analysis of signal decomposition based on regularizing functionals.

\subsection{Semi-inner-product}
In \cite{Lumer61} Lumer introduced the notion of semi-inner-product (s.i.p.), where Giles \cite{Giles67} refined it by asserting
the homogeneity property for both arguments. Semi-inner-products have been used in the analysis of Banach spaces \cite{Bruck1973,sip_book2004} and
in recent years extending Hilbert-space-like concepts in the context of machine-learning and classification \cite{classification_Banach_Der_Lee_2007,zhang2009reproducing,NL_PCA2015}. In general, a s.i.p. is defined for complex-valued functions. Here we restrict ourselves to real-valued functions and follow the definitions of \cite{classification_Banach_Der_Lee_2007}.
\begin{definition}[Semi-inner-product]
\label{def:sip}
Let $(\mathcal{X},\|\cdot\|)$ be a real Banach space. A semi-inner-product on $\mathcal{X}$ is a real function
$[u,v]$ on $\mathcal{X} \times \mathcal{X}$ with the properties:
\begin{enumerate}
\item (Linearity in the first argument) $$[u_1 + u_2,v] = [u_1,v] + [u_2,v],$$
\item (Homogeneity in the first argument) $$[\alpha u,v] = \alpha [u,v],$$
\item (Norm-inducing) $$[u,u] = \|u\|^2,$$
\item (Cauchy-Schwarz inequality) $$[u,v] \le \|u\| \|v\|,$$
\item (Homogeneity in the second argument) $$[u,\alpha v] = \alpha [u,v].$$
\end{enumerate}
\end{definition}
Giles \cite{Giles67} has added the fifth property (Homogeneity in the second argument), arguing that in the case of norms this
does not impose additional restrictions and increases the structure. In the proposed generalizing to functionals, in some cases this condition will be omitted.
In \cite{Giles67} a semi-inner-product for $L^p$ norms $\|u\|_{L^p}=\left( \int_{\Omega}|u(x)|^p dx\right)^{1/p}$, $1 < p < \infty$
was proposed
\begin{equation}
\label{eq:Lp_sip}
 [u,v] :=  \int_{\Omega}u(x)v(x)|v(x)|^{p-2}dx \frac{1}{\|v\|_{L^p}^{p-2}}.
\end{equation}

\subsection{One-homogeneous functionals}
Let $J(u)$ be a proper, convex, lower semi-continuous regularization functional $J: {\mathcal{X}} \rightarrow \mathbb{R}^+ \cup \{\infty\}$ defined on Banach space ${\mathcal{X}}$. 
For $J$ which is a one homogeneous functional we have
\begin{equation}
\label{eq:J1hom}
 J(\alpha u) = |\alpha|J(u),  \,\, \alpha \in \mathbb{R}.
\end{equation}
 We assume that $J(u) > 0 $ for $u \in \mathcal{X} \setminus \{0\}$ (as done for instance in \cite{spec_one_homog15}). This can be achieved by choosing $\mathcal{X}$ restricted in the right way (note that the null-space of a convex one-homogeneous functional is a linear subspace of $\mathcal{X}$, \cite{Benning_Burger_2013}). E.g. in the case of total variation regularization we would consider the subspace of functions with vanishing mean value. The general case can be reconstructed by adding appropriate nullspace components.

Let $p(u) \in \mathcal{X}^*$ (where $\mathcal{X}^*$ is the dual space of $\mathcal{X}$) belong to the \emph{subdifferential} of $J(u)$, defined by:
\begin{equation}
\label{eq:subdif}
\begin{array}{l}
\p J(u):= \\
\left\{p(u) \in \mathcal{X}^* \,|\, J(v)-J(u) \ge \< p(u),v-u \>, \forall v \in \mathcal{X}\right\}.
\end{array}
\end{equation}
We denote $p(u) \in \p J(u)$, where an element $p(u)$ is referred to as a \emph{subgradient}.
For convex one homogeneous functionals it is well known \cite{Ekeland_Temam1999} that
\begin{equation}
\label{eq:up}
J(u) =  \< u,p(u) \>, \forall p(u) \in \p J(u).
\end{equation}
And also, for all  $p(u) \in \p J(u)$, $\mathbb{R} \ni \alpha \ne 0$, we have
\begin{equation}
\label{eq:p_alpha}
\sgn(\alpha) p(u) \in \p J(\alpha u),
\end{equation}
where $\sgn(\cdot)$ is the signum function.
From \eqref{eq:subdif} and \eqref{eq:up} we have that an element in the subdifferential of one-homogeneous functionals admits the following inequality:
\begin{equation}
\label{eq:subdif1hom_classic}
J(v) \ge \< v, p(u) \>, \forall p(u) \in \p J(u), \,v \in \mathcal{X}.
\end{equation}
In later sections we need a slight extension of this property, where the bound is with respect to the magnitude of the right-hand-side.
Since $J(-v)=J(v)$ we can also plug $-v$ in \eqref{eq:subdif1hom_classic} and get the bound $J(v)\ge -\< v, p(u) \>$, hence
\begin{equation}
\label{eq:subdif1hom}
J(v) \ge |\< v, p(u) \>|, \forall p(u) \in \p J(u), \,v \in \mathcal{X}.
\end{equation}

One-homogeneous functionals also admit the triangle inequality:
\begin{equation}
\label{eq:triangle}
J(u+v) \le J(u) + J(v).
\end{equation}
This can be shown by $J(u+v)=\<u+v,p(u+v)\>=\<u,p(u+v)\>+\<v,p(u+v)\>$
and using \eqref{eq:subdif1hom_classic} we have $J(u) \ge \<u,p(u+v)\>$ and $J(v) \ge \<v,p(u+v)\>$.

\subsection{Nonlinear Eigenfunctions}
Let us begin by stating the nonlinear eigenvalue problem induced by
a convex functional.
\begin{definition}[Eigenfunctions and eigenvalues induced by $J(u)$]
An eigenfunction $u$ induced by the functional $J(u)$ admits the following equation,
\begin{equation}
\label{eq:eigenfunction}
\lambda u  \in \partial J(u),
\end{equation}
where $\lambda \in \mathbb{R}$ is the corresponding eigenvalue.
\end{definition}

The analysis of eigenfunctions related to non- quadratic convex functionals was mainly concerned with the total variation (TV) regularization.
In the analysis of the variational TV denoising, i.e. the ROF model from \cite{rof92}, Meyer \cite{Meyer[1]} has shown an explicit solution for the case of a disk (an eigenfunction of TV), quantifying explicitly the loss of contrast
and advocating the use of $TV-G$ regularization.
Within the extensive studies of the TV-flow \cite{tvFlowAndrea2001,andreu2002some,bellettini2002total,tvf_giga2010} eigenfunctions of TV (referred to as \emph{calibrable sets}) were analyzed and explicit solutions were given for several cases of eigenfunction spatial settings. In \cite{iss} an explicit solution of a disk for the inverse-scale-space flow is presented, showing its instantaneous appearance at a precise time point related to its radius and height.


\subsubsection*{Generalizing a classical TV result}
In \cite{tv_flow} a connection between the eigenvalue $\lambda$ and
the perimeter to area ratio is established for the total-variation (TV) case.
Let us recall this relation. The TV functional is defined by
\begin{equation}
\label{eq:Jtv}
	J_{TV}(u) = \sup_{\|\varphi\|_{L^\infty(\Omega)}\le 1} \int_\Omega u \Div \varphi dx,
\end{equation}
with $\varphi \in C^\infty_0$.
For a convex set $A \subset \mathbb{R}^2$ let $f_A$ be the indicator
function of $A$ where $f(x)=1$ for any $x\in A$ and zero otherwise.
If $f_A$ is an eigenfunction (admits Eq. \eqref{eq:eigenfunction}) with
respect to the TV functional then
\begin{equation}
\label{eq:lam_tv}
\lambda = \frac{P(A)}{|A|},
\end{equation}
with $P(A)$ the perimeter of the set A and $|A|$ its area.
The proof is quite elaborated and is based on geometrical arguments.

Using convex analysis arguments, one can generalize this result to any function $f$ which
is an eigenfunction and any one-homogeneous convex functional.
If $J$ is a convex one-homogeneous functional and $f$ admits \eqref{eq:eigenfunction} then
\begin{equation}
\label{eq:lam_1hom}
\lambda = \frac{J(f)}{\|f\|^2}.
\end{equation}
This can be easily shown by using \eqref{eq:up} and \eqref{eq:eigenfunction} having
$$ J(f) = \<p(f),f \> = \<\lambda f, f \> = \lambda \| f \|^2.$$
Eq. \eqref{eq:lam_tv} is thus a special case of \eqref{eq:lam_1hom} where $f$ is an indicator function of a set, therefore $\|f\|^2=|A|$ and for $J$ being TV  we have $P(A)=J_{TV}(f)$, through the coarea formula.

\subsubsection*{Positive eigenvalues}
For the one homogeneous case we readily get from Eq. \eqref{eq:lam_1hom} that all eigenvalues are non-negative, $\lambda \ge 0$. The eigenvalues are strictly positive when $f$ is not in the null-space of $J$ (thus $J(f)>0$) and $\|f\| > 0$ .
A broader statement for general convex functionals is given in \cite{Gilboa_ef_flow} where it is shown that for any eigenfunction $f$, in which
$J(f)>J(0)$, $\|f\| > 0$, we have $\lambda > 0$.

\off{
This can be shown by using Eq. \eqref{eq:subdif} with $v=0$, yielding
$$ J(0) - J(u) \ge \langle p(u),-u \rangle. $$
For $p(u)=\lambda u$ we obtain $ J(u)-J(0) \le \lambda \| u\|^2$, and for any $u \in \mathcal{X}$ with $J(u)>J(0)$ we have
\begin{equation}
\label{eq:lam_pos_conv}
0 < \frac{J(u)-J(0)}{\|u\|^2} \le \lambda.
\end{equation}

{\bf Measure of affinity to eigenfunctions.} For any function $u \in \mathcal{X}$ we would like a quantitative measure
of how close the function is to an eigenfunction. The following measure is suggested in \cite{Gilboa_ef_flow} for one-homogeneous functionals:
\begin{equation}
\label{eq:ef_aff}
\textrm{A}_J(u):=\frac{J(u)}{\| p(u)\| \cdot \|u \|}.
\end{equation}
Using the relation $J(u)=\langle p(u),u\rangle$, the fact that $J(u)\ge0$ and the Cauchy-Schwarz inequality we have two important properties:
$$0 \le \textrm{A}_J(u) \le 1,$$
and
$$\textrm{A}_J(u) = 1 \textrm{ iff } p(u)=\lambda u.$$
That is, the measure is 1 for all eigenfunctions and only for them.
The measure then has a graceful degradation from 1 to 0.
} 

\off{
Let us define the projection of $u$ onto the plane orthogonal to $p(u)$:
$$w:=  u - \frac{\langle u,p(u) \rangle}{\|p(u)\|^2}p(u).$$
Then $\textrm{A}_J(u)$ decreases as $\|w\|$ increases, where for eigenfunctions $\|w\|=0$.
}

\subsection{The TV Transform}

In \cite{Gilboa_spectv_SIAM_2014} a generalization of eigenfunction analysis to the total-variation case was proposed in the following way.
Let $u(t;x)$ be the TV-flow solution \cite{tv_flow} or the gradient descent of the total variation energy $J_{TV}(u)$,  with initial condition $f(x)$:
\begin{equation}
\label{eq:TVflow}
	\partial_t u = - p, \qquad p \in \p J_{TV}(u), \qquad u(t=0)=f(x),
\end{equation}
where $J_{TV}$ is given in \eqref{eq:Jtv}.
The TV spectral transform is defined by
\begin{equation}
\label{eq:phi}
\phi(t;x) := t\partial_{tt}u (t;x),
\end{equation}
where $\partial_{tt}u$ is the second time derivative of the solution $u(t;x)$ of the TV flow \eqref{eq:TVflow}.
For $f(x)$ admitting \eqref{eq:eigenfunction}, with a corresponding eigenvalue $\lambda$, one obtains
a gradient flow Eq. \eqref{eq:TVflow} with a solution
\begin{equation}
\label{eq:u_ef}
u(t,x) = (1-\lambda t)^+ f(x),
\end{equation}
where $(q)^+ = q$ if $q > 0$ and 0 otherwise.
The spectral response becomes
\begin{equation}
\label{eq:phi_ef}
\phi(t;x)=\delta(t-1/\lambda)f(x),
\end{equation}
where $\delta(\cdot)$ denotes a Dirac delta distribution.

In the general case, $\phi$ yields a continuum multiscale representation of the image, generalizing structure-texture decomposition
methods like \cite{Meyer[1],Luminita[2],agco06}. For simplicity we assume signals with zero mean $\bar{f} = \frac{1}{\Omega}\int_\Omega f(x)dx = 0$.
One can reconstruct the original image by:
\begin{equation}
\label{eq:tv_recon}
f(x) = \int_0^\infty \phi(t;x) dt.
\end{equation}
Given a transfer function $H(t)\in \mathbb{R}$, image filtering can be performed by
\begin{equation}
\label{eq:tv_filt}
f_H(x) := \int_0^\infty H(t)\phi(t;x) dt.
\end{equation}
Simple useful filters are ones which either retain or diminish completely scales up to some cutoff scale.
The (ideal) low-pass-filter (LPF) can be defined by Eq. \eqref{eq:tv_filt} with $H(t)=1$ for $t \ge t_c$ and 0 otherwise, or
\begin{equation}
\label{eq:lpf}
LPF_{t_c}(f) := \int_{t_c}^\infty \phi(t;x) dt.
\end{equation}
Its complement, the (ideal) high-pass-filter (HPF), is defined by
\begin{equation}
\label{eq:hpf}
HPF_{t_c}(f) := \int_{0}^{t_c} \phi(t;x) dt.
\end{equation}
Similarly, band-(pass/stop)-filters are filters with low and high cut-off scale parameters ($t_1 < t_2$)
\begin{equation}
\label{eq:bpf}
BPF_{t_1,t_2}(f) := \int_{t_1}^{t_2} \phi(t;x) dt,
\end{equation}
\begin{equation}
\label{eq:bsf}
BSF_{t_1,t_2}(f) := \int_0^{t_1}\phi(t;x) dt + \int_{t_2}^\infty \phi(t;x) dt.
\end{equation}
The spectrum $S_f(t)$ corresponds to the amplitude of each scale of the input $f$:
\begin{equation}
\label{eq:S}
S_f(t) := \|\phi(t;x)\|_{L^1(\Omega)} = \int_\Omega |\phi(t;x)|dx.
\end{equation}

\begin{figure}[htb]
\begin{center}
\begin{tabular}{ cc }
\includegraphics[width=40mm]{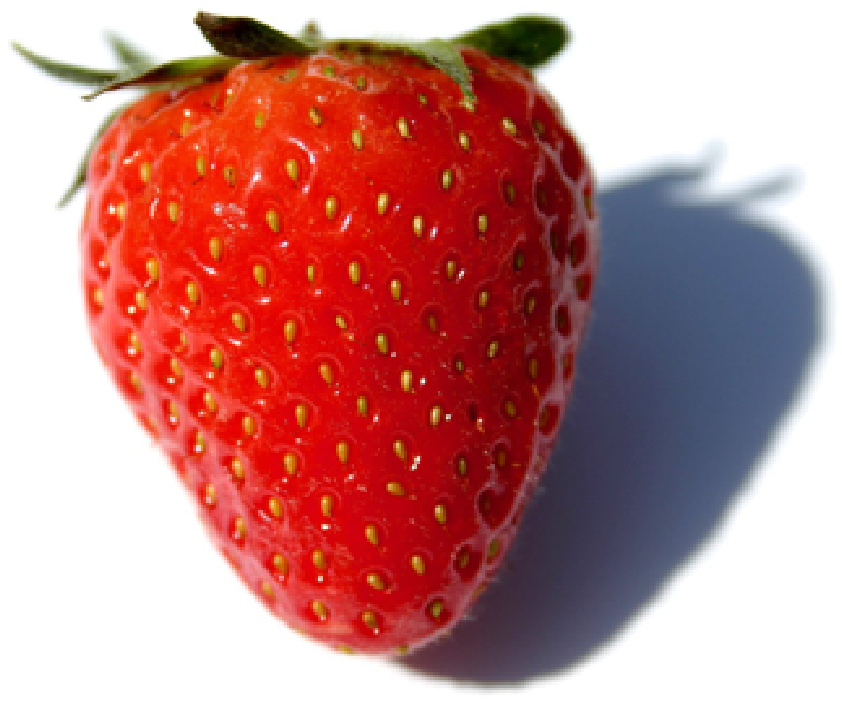}&
\includegraphics[width=40mm]{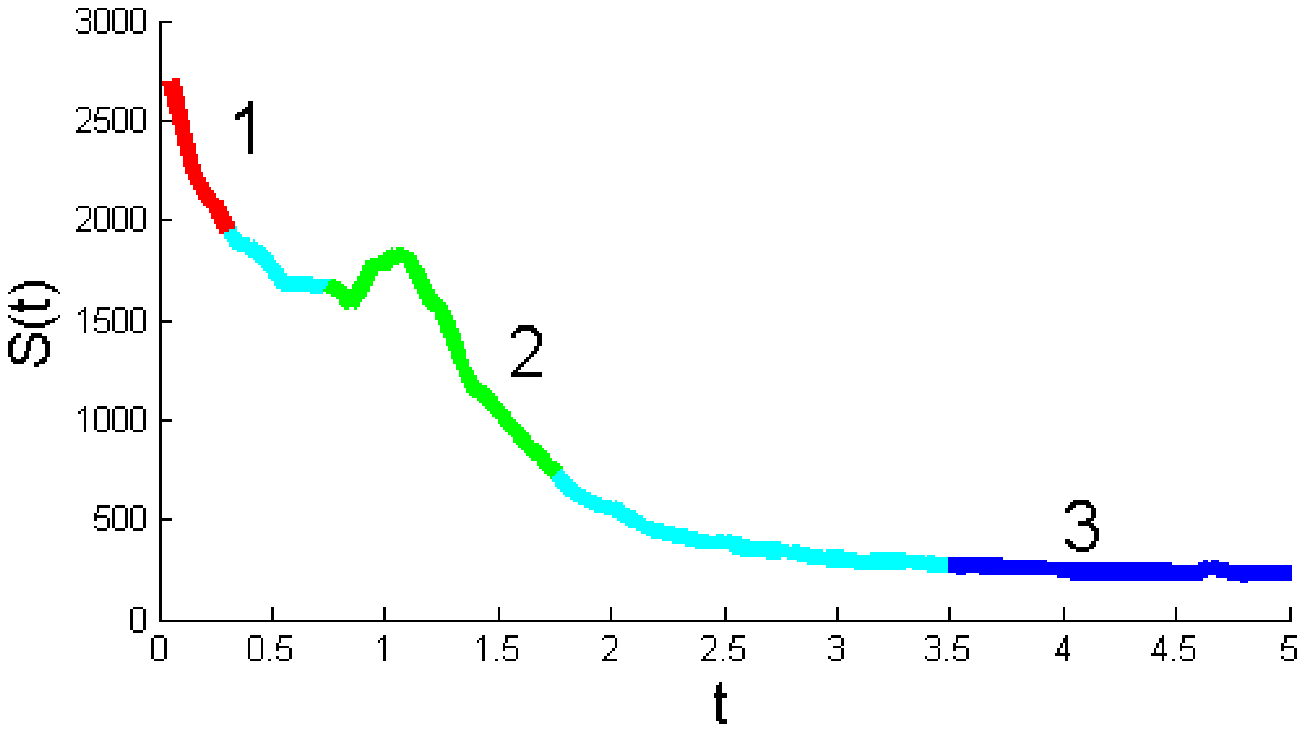}\\
Input $f$ &  $S_1(t)$\\
\includegraphics[width=40mm]{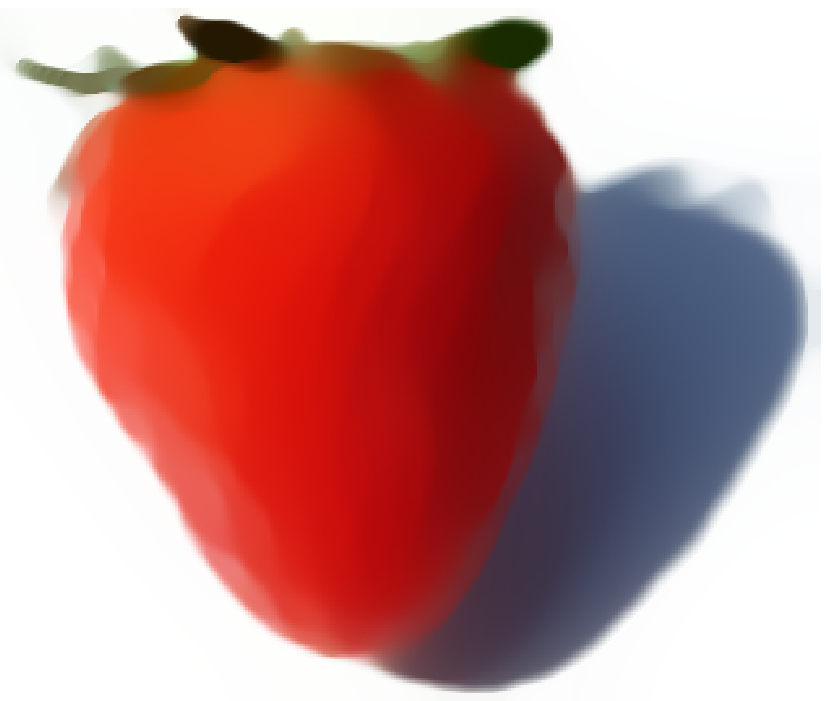}&
\includegraphics[width=40mm]{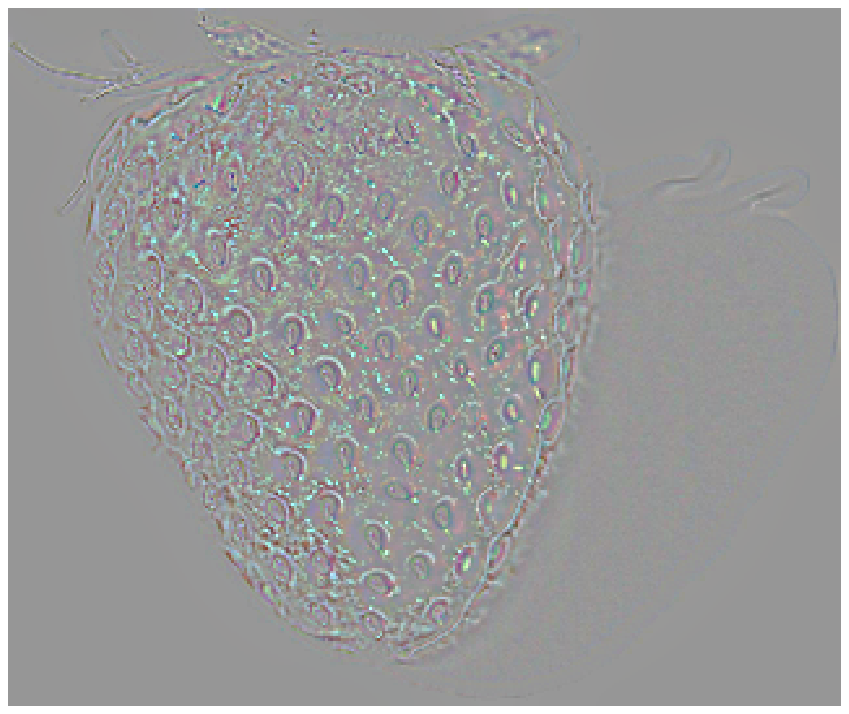}\\
Low-pass & High-pass\\
\includegraphics[width=40mm]{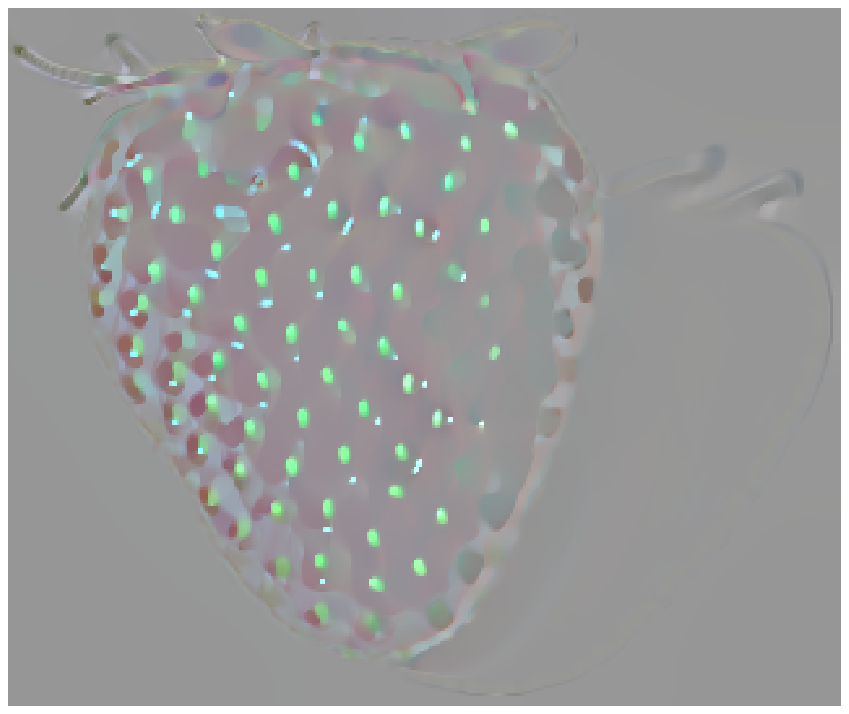}&
\includegraphics[width=40mm]{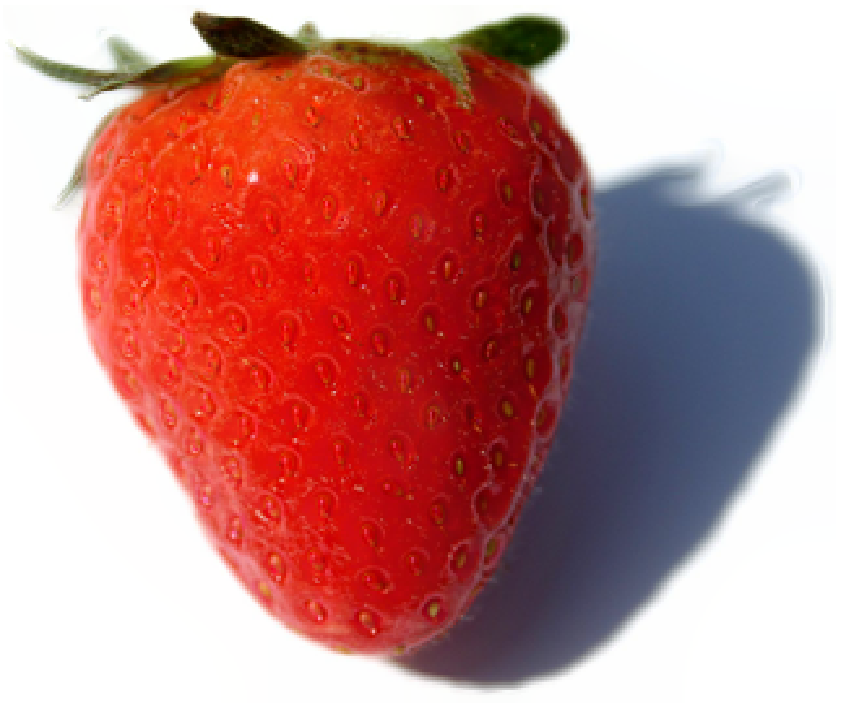}\\
Band-pass & Band-stop \\
\end{tabular}
\caption{Total-variation spectral filtering example. The input image (top left) is decomposed into its $\phi(t)$ components, the corresponding spectrum
$S_1(t)$ is on the top right. Integration of the $\phi$'s over the $t$ domains 1, 2 and 3 (top right) yields high-pass, band-pass and low-pass filters, respectively. The band-stop filter (bottom right) is the complement integration domain of region 2. Taken from \cite{spec_one_homog15}.}
\label{fig:strawberry}
\end{center}
\end{figure}

In Fig. \ref{fig:strawberry} an example of spectral TV processing is shown with the response of the four filters defined above in Eqs. \eqref{eq:lpf} through \eqref{eq:bsf}.

\subsection{Generalized Transform}
In \cite{spec_one_homog15} the spectral TV framework was generalized in several ways. First
the theory was extended to a wider class of one-homogeneous functionals.

For the general gradient flow of a one-homogeneous functional $J$,
\begin{align}
\label{eq:grad_flow}
\partial_t u(t) = -p(t), \qquad p(t) \in \partial J(u(t)), \ u(0)=f,
\end{align}
the spectral transform $\phi(t)$, the eigenfunction response, the reconstruction and the filtering, Eqs. \eqref{eq:phi}, \eqref{eq:u_ef}, \eqref{eq:phi_ef}, \eqref{eq:tv_recon}, \eqref{eq:tv_filt} all generalize
in a straightforward manner, retaining the same expressions.

\off{
An eigenfunction which admits \eqref{eq:eigenfunction} has the gradient flow solution
$$u(t;x) = (1-\lambda t)_+ f(x),$$
yielding
$$\phi_*(t) = f \delta_{\frac{1}\lambda}(t).$$
}

A new spectrum was defined by
\begin{equation}
\label{eq:newSdefinition}
	S_2(t) = t \sqrt{ \frac{d^2}{dt^2}J(u(t)) } = \sqrt{\langle \phi(t), 2 t p(t)} \rangle,
\end{equation}
for which an analogue of the Parseval identity can be derived
$$ 	\Vert f \Vert^2 = \int_0^\infty S_2(t)^2  ~dt.$$

An orthogonality property was shown
\begin{equation}
\label{eq:phi_u}
	\langle \phi(t), u(t) \rangle = 0, \,\, \forall t>0.
\end{equation}
An overview of these ideas with relations to some classical signal processing methods are presented in \cite{jmivPreprint}.

With the preliminary settings and definitions in place we can now continue to the
main contributions of the paper concerning generalized s.i.p's.

\section{A semi-inner-product for convex functionals}
Let us define a semi-inner-product for convex functionals, in a similar manner to Definition \ref{def:sip}.
As we will show later, a function which admits the properties below may not be unique. Therefore, in
a similar manner to the subdifferential, we allow the semi-inner-product to be a set of possibly more than one element.
We denote by $[u,v]_J$ an element and by $\{ [u,v]_J \}$ the set of admissible s.i.p.'s.
We will later see for the one-homogeneous case that when a specific subgradient of the second argument is chosen the s.i.p. is unique.
\begin{definition}[Semi-inner-product of a convex functional, partial homogeneity]
\label{def:fsip}
Let $J$ be a convex functional $J : \mathcal{X} \to \mathbb{R}^+ \cup \{ \infty \}$ defined on a Banach
space $\mathcal{X}$. 
A semi-inner-product with partial homogeneity on $\mathcal{X}$ is a real function $[u,v]_J$ on $\mathcal{X} \times \mathcal{X}$ with the properties:
\begin{enumerate}
\item (Linearity in the first argument)
$$[u_1 + u_2,v]_J = s + q,\, s \in \{ [u_1,v]_J\},\, q \in \{ [u_2,v]_J\}.$$
\item (Homogeneity in the first argument) $$[\alpha u,v]_J \in \{ \alpha [u,v]_J \}, \, \alpha \in \mathbb{R}.$$
\item (Functional-inducing) $$[u,u]_J = J^2(u).$$
\item (Cauchy-Schwarz-type inequality)
\begin{equation}
\label{eq:CS-type}
\sqrt{|[u,v]_J[v,u]_J|} \le J(u)J(v).
\end{equation}
\end{enumerate}
\end{definition}

A stricter definition, with homogeneity in both arguments is defined by
\begin{definition}[Semi-inner-product of a convex functional, full homogeneity]
\label{def:fsip_full}
Following the same notations of Def. \ref{def:fsip},  $[u,v]_J$ is a semi-inner-product with full homogeneity if it admits all the properties of Def. \ref{def:fsip} and in addition:
\begin{enumerate}
\setcounter{enumi}{4}
\item (Homogeneity in the second argument) $$[u,\alpha v]_J \in \{ \alpha [u,v]_J \}.$$
\end{enumerate}
\end{definition}

\subsection{Semi-inner-product formulations}
It can be verified that for functionals of the form
\begin{equation}
\label{eq:J_H}
J_{\mathcal{H}}(u)=\|u\|_{\mathcal{H}}^2,
\end{equation}
with $\{\|\cdot\|_{\mathcal{H}},\<\cdot,\cdot \>_{\mathcal{H}}\}$ a Hilbert-space norm and inner-product, respectively, a semi-inner-product in the sense of Def. \ref{def:fsip} is:
\begin{equation}
\label{eq:sip_H}
[u,v]_{J_\mathcal{H}}:= \<u,v\>_{\mathcal{H}} \|v\|^2_{\mathcal{H}}.
\end{equation}

However, our main focus of the paper is devoted to functionals not based on a Hilbert-space but
on smoothing, discontinuity preserving functionals such as the total-variation or the total-generalized-variation. Those functionals are extremely useful in processing images and many other
types of signals with inherent discontinuities, such as depth-maps or optical-flow fields.
Those functionals are one-homogeneous and therefore a full homogeneity semi-inner-product can be
defined.

\begin{theorem}
\label{th:sip_1hom}
Let $J$ be a convex one-homogeneous functional and $p(v) \in \partial J(v)$ a subgradient. Then a corresponding semi-inner-product with full homogeneity in the sense of Def. \ref{def:fsip_full} is
\begin{equation}
\label{eq:sip_1hom}
[u,v]^{p(v)}_J:= \<u,p(v)\>J(v), 
\end{equation}
where $\<\cdot,\cdot\>$ denotes the $L^2$ inner product.
\end{theorem}
\begin{proof}
Linearity and homogeneity in the first argument are straightforward consequences of using the $L^2$ inner product.
We now want to show the property of homogeneity in the second argument.
We use Eqs. \eqref{eq:J1hom} and \eqref{eq:p_alpha} to have
$ p(v) \in \p J(v)$ and $ p(\alpha v) \in \p J(\alpha v)$ with the relation
$p(\alpha v)=\sgn(\alpha)p(v)$ and therefore
\begin{displaymath}
\begin{array}{ll}
[u,\alpha v]^{p(\alpha v)}_J & = \<u,p(\alpha v)\>J(\alpha v) \\
& =\<u,\sgn(\alpha)p(v)\>|\alpha|J(v) \\
& = \alpha[u,v]^{p(v)}_J \in \{ \alpha [u,v]_J \}.
\end{array}
\end{displaymath}
Using \eqref{eq:up} we get $[u,u]^{p(u)}_J=\<u,p(u)\>J(u)=J^2(u)$.
Finally for the Cauchy-Schwarz property, using \eqref{eq:subdif1hom} we have
$\forall p(u) \in \p J(u)$, $J(v) \ge |\<v,p(u)\>|$ and $\forall p(v) \in \p J(v)$, $J(u) \ge |\<u,p(v)\>|$, therefore
$$ |[u,v]^{p(v)}_J| = |\<u,p(v)\>| J(v) \le J(u)J(v)$$
and also
$$ |[v,u]^{p(u)}_J| = |\<v,p(u)\>|J(u) \le J(v)J(u).$$
\end{proof}

\off{  
Based on the triangle inequality \eqref{eq:triangle} a useful bound is
\begin{equation}
\label{eq:u_uv_bound}
|[u,u+v]_J| \le J^2(u)+J(u)J(v).
\end{equation}
This is shown by using the definition $|[u,u+v]_J| = |\<u,p(u+v) \>J(u+v)|$. From
\eqref{eq:subdif1hom} and \eqref{eq:triangle} we have $ |\<u,p(u+v)\>|  \le J(u)$ and
$J(u+v) \le J(u)+J(v)$, respectively.
}

As noted in the proof, for the one-homogeneous s.i.p. a classical Cauchy-Schwarz inequality holds
\begin{equation}
\label{eq:CS}
|[u,v]_J^{p(v)}| \le J(u)J(v).
\end{equation}

As an example, let us take the $L^q$ norm, $J_{L^q}(u)=\|u\|_{L^q}$, for $1 < q < \infty$.
Then $p(u) = |u|^{q-2}u \|u\|_{L^q}^{1-q}$ and Eq. \eqref{eq:sip_1hom} coincides with \eqref{eq:Lp_sip}.

\subsection{Generalized notions of angle and orthogonality}
With the s.i.p. one can define an angle between functions $u$ and $v$.
For brevity, we will omit the superscript $p(v)$ when the context is clear.
In the one-homogeneous case, using the above inequality, we can define the angle between $u$ and $v$ by
\begin{equation}
\label{eq:angle}
\begin{array}{ll}
\textrm{angle}(u,v) & := \cos^{-1}\left( \frac{[u,v]_J}{J(u)J(v)}\right). 
\end{array}
\end{equation}
Note that there is no symmetry in the above definition, so in general $\textrm{angle}(u,v) \ne \textrm{angle}(v,u)$.

For a symmetric angle expression, there are two main options, an algebraic mean,
\begin{equation}
\label{eq:angle_sym1}
\begin{array}{ll}
\textrm{angle}_{sym-a}(u,v) & := \cos^{-1}\left( \frac{\frac{1}{2}\left([u,v]_J+[v,u]_J\right)}{J(u)J(v)}\right), 
\end{array}
\end{equation}
and a geometric mean (which also applies for the general convex case, in which the inequality of \eqref{eq:CS-type} holds),
\begin{equation}
\label{eq:angle_sym2}
\begin{array}{ll}
\textrm{angle}_{sym-g}(u,v) & := \cos^{-1}\left( \frac{\mathcal{S}\left([u,v]_J,[v,u]_J\right)}{J(u)J(v)}\right), 
\end{array}
\end{equation}
where $\mathcal{S}(a,b):=\sgn(ab)\sqrt{|ab|}$ is a signed square-root.

Orthogonality of two functions can be expressed as having an angle of $\frac{\pi}{2}$ between them. In the case
of the nonsymmetric angle of \eqref{eq:angle} we refer to $u$ as \emph{orthogonal} to $v$ if $0 \in \{[u,v]_J\}$
and to $v$ as \emph{orthogonal} to $u$ if $0 \in \{[v,u]_J\}$.

\begin{definition}[Full orthogonality (FO)]
\label{def:full_orth}
$(u,v)$ are fully orthogonal if $0 \in \{[u,v]_J\}$ and $0 \in \{[v,u]_J\}$.
\end{definition}

\subsection{A semi-inner-product of degree $q$}
A slight generalization of the s.i.p. defined above is a \emph{semi inner product of degree $q$}. 
Essentially the norm and Cauchy-Schwarz properties are raised to the $q$'s power. The formal definition
is as follows.
\begin{definition}[Semi-inner-product of degree $q$ of a convex functional]
\label{def:qsip}
Let $J$ be a convex functional $J : \mathcal{X} \to \mathbb{R}^+ \cup \{ \infty \}$ defined on a Banach
space $\mathcal{X}$ 
A semi-inner-product of degree $q$ on $\mathcal{X}$ is a real function $[u,v]_{J,q}$ on $\mathcal{X} \times \mathcal{X}$ with the properties:
\begin{enumerate}
\item (Linearity in the first argument) $$[u_1 + u_2,v]_{J,q} = s + q,\,
s \in \{ [u_1,v]_{J,q}\},\, q \in \{ [u_2,v]_{J,q}\},$$
\item (Homogeneity in the first argument) $$[\alpha u,v]_{J,q} \in \{ \alpha [u,v]_{J,q} \}, \, \alpha \in \mathbb{R},$$
\item (Functional-inducing) $$[u,u]_{J,q} = J^{2q}(u),$$
\item (Cauchy-Schwarz-type inequality) $$\sqrt{|[u,v]_{J,q} \cdot [v,u]_{J,q}|} \le J^q(u)J^q(v).$$
\end{enumerate}
\end{definition}
We examine more closely the s.i.p. of degree half ($q=\frac{1}{2}$) abbreviated h.s.i.p.
For brevity we denote a special symbol for it
$$ \lfloor u,v \rfloor_J := [u,v]_{J,1/2}. $$
For the h.s.i.p. property 3 in Def. \ref{def:qsip} becomes  $\lfloor u,u \rfloor_J = J(u),$ and property 4 becomes $\lfloor u,v \rfloor_J \lfloor v,u \rfloor_J \le J(u)J(v)$.

In the case of square Hilbert-space functionals, Eq. \eqref{eq:J_H}, we get
$$\lfloor u,v \rfloor_{J_{\mathcal{H}}} = \< u, v \>_{\mathcal{H}} = \frac{[u,v]_{J_\mathcal{H}}}{J_\mathcal{H}(v)}.$$

We will now examine the one-homogeneous case.
\begin{proposition}
\label{th:hsip_1hom}
Let $J$ be a convex one-homogeneous functional and $p(v) \in \partial J(v)$ a subgradient. Then a corresponding semi-inner-product of degree $1/2$ in the sense of Def. \ref{def:qsip} is
\begin{equation}
\label{eq:hsip_1hom}
\lfloor u,v \rfloor_J^{p(v)}:= \<u,p(v)\>.
\end{equation}
\end{proposition}
\begin{proof}
The proof is mainly similar to the one of Th. \ref{th:sip_1hom}. For the third property we use Eq. \eqref{eq:up}
and for the fourth property, using \eqref{eq:subdif1hom}, we have
$|\lfloor u,v \rfloor_J^{p(v)}| \le J(u)$ and $|\lfloor v,u \rfloor_J^{p(v)}| \le J(v)$.
\end{proof}
Note that the s.i.p. of \eqref{eq:sip_1hom} is simply the h.s.i.p.
multiplied by $J(v)$,
\begin{equation}
\label{eq:sip_hsip}
[u,v]_J^{p(v)} = \lfloor u, v \rfloor_J^{p(v)} J(v).
\end{equation}

Following Eqs. \eqref{eq:up}, \eqref{eq:p_alpha}, \eqref{eq:subdif1hom} we have for the one- homogeneous h.s.i.p. the following properties:
\begin{equation}
\label{eq:hsip_J}
\lfloor u, u \rfloor_J = J(u),
\end{equation}

\begin{equation}
\label{eq:hsip_alpha}
\lfloor u,\alpha v \rfloor_J \in \{ \sgn(\alpha) \lfloor u,v \rfloor_J \}, \qquad \mathbb{R} \ni \alpha \ne 0,
\end{equation}

\begin{equation}
\label{eq:hsip_ineq}
|\lfloor u,v \rfloor_J| \le \lfloor u,u \rfloor_J = J(u), \qquad \forall v \in \mathcal{X}.
\end{equation}

\subsection{Relation to Bregman distance}
We will now show the close connection between the Bregman distance (also called Bregman divergence) and the s.i.p.
in the one-homogeneous case.

Let us first recall the Bregman distance definition \cite{bre67}.
For a convex functional $J$ and a subgradient $p(v) \in \partial J(v)$,
the (generalized) Bregman distance is
\begin{equation}
\label{eq:breg_dist}
D_J^{p(v)}(u,v) := J(u) - J(v) - \langle p(v),u-v \rangle.
\end{equation}
This is not necessarily a distance in the standard sense, as it is not necessarily symmetric and does not admit the triangle inequality, however it is guaranteed to be non-negative and it is identically zero for $u=v$.
For $J$ the square $L^2$ norm we get the Euclidean distance squared,
 $$D_{\|\cdot\|^2}(u,v)=\|u-v\|^2.$$
Other known similarity measures, such as the KL- divergence or the
Mahalanobis distance, can also be derived from \eqref{eq:breg_dist} with appropriate functionals \cite{Clustering_Bregman_div_JMRL2005}.
This measure has been widely used in the theoretical analysis of classification, clustering and convex optimization algorithms, see e.g. \cite{Clustering_Bregman_div_JMRL2005,Censor_multiprojection_Bregman_1994,Metric_learning_JMRL2012,Nemirovski_subgrad_algo_2014}.
Specifically for image processing, a significant branch of studies has
presented iterative variational solutions, new evolution formulations and numerical solvers based on the Bregman distance, especially in relation to
total-variation and other one-homogeneous regularizing functionals
\cite{Bregman_obgxy,iss,SplitBregman2009,Zhang_Burger_Bresson_Osher_NL_Bregman_2010,Ma_Goldfarb_rank_min_2011}, see a recent review of the topic in \cite{Burger_Bregman_IP_2015}.

In the one-homogeneous case we use the relation $J(v)=\langle v,p(v) \rangle$ and the expression in \eqref{eq:breg_dist} simplifies to
\begin{equation}
\label{eq:breg_dist_1hom}
D_J^{p(v)}(u,v)|_{(\textrm{ 1-hom})} = J(u) - \langle p(v),u \rangle.
\end{equation}
It is straightforward in this case to infer the relation to
the s.i.p. and h.s.i.p.,
\begin{equation}
\label{eq:breg_dist_1hom_equiv}
\begin{array}{lll}
D_J^{p(v)}(u,v)|_{(\textrm{ 1-hom})} &= & J(u) - \frac{[u,v]_J^{p(v)}}{J(v)} \\
& = & J(u) - \lfloor u,v \rfloor_J^{p(v)}.
\end{array}
\end{equation}
An interesting interpretation of the Bregman distance is with respect
to the angle between the functions $u$ and $v$,
\begin{equation}
\label{eq:breg_dist_1hom_angle}
D_J^{p(v)}(u,v)|_{(\textrm{ 1-hom})} = J(u)\left(1 - \cos(\textrm{angle}(u,v)) \right),
\end{equation}
with the angle defined in \eqref{eq:angle}.
With this expression we can immediately get the upper and lower bounds
$$ 0 \le D_J^{p(v)}(u,v) \le 2J(u).$$
Moreover, the interpretation of the Bregman distance is of having direct relation to the angle between the functions;
the Bregman distance is zero for zero angle and is monotonically increasing with angle, reaching the maximum at angle$(u,v)=\pi$.

An extension of this relation which applies to the general convex case is not known at this point.
We now define the final notions needed for the decomposition theorem.

\begin{figure}[htb]
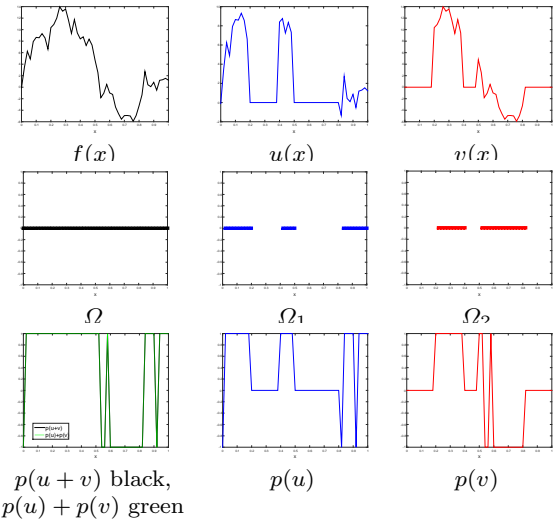

\begin{center}
\begin{tabular}{ ccc }
\includegraphics[width=20mm]{lis_f.eps}&
\includegraphics[width=20mm]{lis_u.eps}&
\includegraphics[width=20mm]{lis_v.eps}\\
$f(x)$ & $u(x)$ & $v(x)$ \\
\includegraphics[width=20mm]{lis_Omega.eps}&
\includegraphics[width=20mm]{lis_Omega1.eps}&
\includegraphics[width=20mm]{lis_Omega2.eps}\\
$\Omega$ & $\Omega_1$ & $\Omega_2$ \\
\includegraphics[width=20mm]{lis_puv1.eps}&
\includegraphics[width=20mm]{lis_pu.eps}&
\includegraphics[width=20mm]{lis_pv.eps}\\
$p(u+v)$ black, & $p(u)$ & $p(v)$ \\
$p(u)+p(v)$ green
\end{tabular}
\caption{LIS example for the $L^1$ norm.}
\label{fig:lis}
\end{center}
\end{figure}

\begin{definition}[Linearity in the subdifferential (LIS)]
\label{def:lin_subdif}
$(u,v)$ are linear in the subdifferential if for any
$\mathbb{R} \ni \{\alpha_1, \alpha_2\} \ne 0$ there exist
$p(\alpha_1 u + \alpha_2 v) \in \p J(\alpha_1 u + \alpha_2 v)$,
$p(\alpha_1 u) \in \p J(\alpha_1 u)$,
$p(\alpha_2 v) \in \p J(\alpha_2 v)$,
such that
\begin{equation}
\label{eq:subdif_lin}
p(\alpha_1 u + \alpha_2 v)  = p(\alpha_1 u) + p(\alpha_2 v).
\end{equation}
\end{definition}
(LIS) implies the h.s.i.p. is linear in the second argument. If the pair $(v_1,v_2)$ admit the (LIS) condition then
there exist 3 subgradient elements $p(\alpha_1 v_1 + \alpha_2 v_2)$, $p(\alpha_1 v_1)$, $p(\alpha_2 v_2)$ such that
for all $u \in \mathcal{X}$ we have
\begin{equation}
\label{eq:LIS_hsip}
\begin{array}{l}
\lfloor u,\alpha_1 v_1 + \alpha_2 v_2 \rfloor_J^{p(\alpha_1 v_1 + \alpha_2 v_2)}  \\
= \lfloor u,\alpha_1 v_1 \rfloor_J^{p(\alpha_1 v_1)} + \lfloor u, \alpha_2 v_2 \rfloor_J^{p(\alpha_2 v_2)}.
\end{array}
\end{equation}
This is shown by writing the left-hand-side, according to \eqref{eq:hsip_1hom}, as
$\langle u,p(\alpha_1 v_1 + \alpha_2 v_2) \rangle$
and using \eqref{eq:subdif_lin}.

We give a simple example of two signals admitting (LIS) in the case of $J$ being the $L^1$ norm, for the 1D case within the unit interval $\Omega=[0,1]$.
Let $f(x)$ be a real function in $\Omega$,  $f:\Omega \to \mathbb{R}$. We define the following two functions:
$u(x)=f(x)$ if $x \in [0,0.5)$ and 0 otherwise, $v(x)=f(x)$ if $x \in [0.5,1]$ and 0 otherwise.
Then it can be verified that $u$ and $v$ are (LIS). Any other partition $\Omega_1 \subset \Omega$, $\Omega_2 = \Omega \setminus \Omega_1$ for $u$ and $v$
will produce similar results, see Fig. \ref{fig:lis}.


\begin{definition}[Independent functions]
\label{def:independent}
$(u,v)$ are independent functions if they are fully orthogonal (FO) and linear in the subdifferential (LIS),
according to Def. \ref{def:full_orth} and Def. \ref{def:lin_subdif}, respectively.
\end{definition}

We shall now show that for one-homogeneous functionals, all functions which are (LIS) are also (FO) and are therefore independent.

\begin{proposition}
\label{prop:indep}
Let $J$ be a convex one-homogeneous functional. If the pair $(u,v)$ are (LIS) according to Def. \ref{def:lin_subdif}, then $(u,v)$ are (FO) and therefore are independent (Def. \ref{def:independent}).
\end{proposition}
\begin{proof}
From \eqref{eq:hsip_ineq} we have
$$ J(u) \ge \lfloor  u, u+v \rfloor_J, $$
using (LIS), for some fixed subgradients $p(u+v)$, $p(u)$, $p(v)$, we have
\begin{displaymath}
\begin{array}{lll}
 J(u) \ge \lfloor u, u+v \rfloor_J^{p(u+v)} & = &\lfloor u, u \rfloor_J^{p(u)} + \lfloor u, v \rfloor_J^{p(v)}\\
  &=& J(u) + \lfloor u, v \rfloor_J^{p(v)}.$$
\end{array}
\end{displaymath}
We therefore have $\lfloor u, v \rfloor_J^{p(v)} \le 0$.
On the other hand, taking $\alpha_1=1,\alpha_2=-1$ in Def. \ref{def:lin_subdif} we get that
also $u,-v$ are (LIS). In this case, using \eqref{eq:hsip_alpha}, we reach $\lfloor u, v \rfloor_J^{p(v)} \ge 0$.
We can conclude that $\lfloor u, v \rfloor_J^{p(v)} = 0$ hence
$$\{ [u,v]_J \} \ni  \lfloor u, v \rfloor_J^{p(v)} J(v) = 0.$$
The same arguments hold for the pair $(v,u)$.
\end{proof}

An interesting characteristic of independent functions is that they reach the upper
bound of the triangle inequality (Eq. \eqref{eq:triangle}).
\begin{proposition}
\label{prop:indep_tri}
Let $J$ be a convex one-homogeneous functional. If $(u,v)$ are independent (Def. \ref{def:independent})
then $J(u+v) = J(u) + J(v)$.
\end{proposition}
\begin{proof}
$$
\begin{array}{ll}
J(u+v) & =\lfloor u+v,u+v \rfloor_J^{p(u+v)} \\
       & \stackrel{\text{LIS}}{=} \lfloor u,u \rfloor_J^{p(u)} + \lfloor u,v \rfloor_J^{p(v)} + \lfloor v,u \rfloor_J^{p(u)}  + \lfloor v,v \rfloor_J^{p(v)} \\
       &\stackrel{\text{FO}}{=} \lfloor u,u \rfloor_J^{p(u)}  + \lfloor v,v \rfloor_J^{p(v)}  \\
       &=J(u)+J(v).
\end{array}
$$
\end{proof}



\subsection{S.I.P. for eigenfunctions}
For eigenfunctions, which admit \eqref{eq:eigenfunction}, things simplify considerably, where an element in the s.i.p. is basically
a weighted $L^2$ inner product.

For $\lambda u \in \p J(u)$ we get, as in \eqref{eq:lam_1hom},
$$ J(u)  =  \< u,p(u) \>  = \< u,\lambda u \> = \lambda \| u \|^2,$$
where $\|\cdot\|$ is the $L^2$ norm.
For semi-inner-products we will often use the subgradient element corresponding
to the eigenfunction, this will be denote by a superscript $\lambda_v v$.
We therefore have the following relations for the s.i.p and h.s.i.p:
For the s.i.p., for any $u \in \mathcal{X}$, $\lambda_v v \in \p J(v)$,
\begin{equation}
\label{eq:sip_ef}
\lambda_v^2 \< u,v \>\|v\|^2 = [u,v]_J^{\lambda_v v} \in \{ [u,v]_J \},
\end{equation}
and for the h.s.i.p.,
\begin{equation}
\label{eq:hsip_ef}
\lambda_v \< u,v \> = \lfloor u,v \rfloor_J^{\lambda_v v} \in \{ \lfloor u,v \rfloor_J \}.
\end{equation}

Another consequence is related to orthogonality.
\begin{proposition}
\label{prop:ef_ortho}
\begin{enumerate}
\item For any $u \in \mathcal{X}$, $\lambda_v v \in \p J(v)$, $\lambda_v>0$, $\|v\|>0$,
$$[u,v]_J^{\lambda_v v}= 0 \,\, \textrm{ iff } \,\, \<u,v \>=0.$$
\item For $p(u)=\lambda_u u \in \p J(u)$, $p(v)=\lambda_v v \in \p J(v)$, $\lambda_u,\lambda_v>0$,
$\|u\|,\|v\|>0$,
the following statements are identical:
\begin{enumerate}
    \item $[u,v]_J^{\lambda_v v} = 0$,
    \item $[v,u]_J^{\lambda_u u} = 0$,
    \item $\lfloor u,v \rfloor_J^{\lambda_v v} = 0$,
    \item $\lfloor v,u \rfloor_J^{\lambda_u u} = 0$,
    \item $\lfloor u,v \rfloor_J^{\lambda_v v} + \lfloor v,u \rfloor_J^{\lambda_u u} = 0$,
    \item $\< u,v \> = 0$.
\end{enumerate}
\end{enumerate}
\end{proposition}
\begin{proof}
The first part is an immediate consequence of Eq. \eqref{eq:sip_ef}.
For the second part, let us write the equivalent of statements (a) through (e):
\begin{itemize}
    \item[(A)] $[u,v]_J^{\lambda_v v} = \lambda_v^2 \< u,v \>\|v\|^2 $,
    \item[(B)] $[v,u]_J^{\lambda_u u} = \lambda_u^2 \< v,u \>\|u\|^2 $,
    \item[(C)] $\lfloor u,v \rfloor_J^{\lambda_v v} = \lambda_v \< u,v \>$,
    \item[(D)] $\lfloor v,u \rfloor_J^{\lambda_u u} = \lambda_u \< v,u \>$,
    \item[(E)] $\lfloor u,v \rfloor_J^{\lambda_v v} + \lfloor v,u \rfloor_J^{\lambda_u u} = (\lambda_v + \lambda_u) \< u,v \>$.
\end{itemize}
We observe that in the case where both $u$ and $v$ are eigenfunctions all expressions reduce to the $L^2$ inner product up to a strictly positive multiplicative
factor and are therefore identical when $\< u,v \> = 0$.
\end{proof}

\section{Decomposition}

Let $f_1, f_2$ be two functions in $\mathcal{X}$ and $f = f_1 + f_2$.
Naturally a decomposition from a single measurement $f$ into two signals $f_1$ and $f_2$ is not possible in general.
One should use some a priori knowledge and assumptions on the signals (depicted in the choice of the regularizer $J$).
A classical decomposition problem is how and under what conditions we can decompose $f$ into $f_1$ and $f_2$.
This issue is significant in signal processing, for instance when $f_1$ is the signal and $f_2$ is noise or for
structure-texture decomposition, where $f_1$ is structure and $f_2$ is texture (assumed to be additive).
We will try to give an answer to this using the spectral filtering technique and conditions from the above framework.


We can now state a sufficient condition for spectral filtering to perfectly decompose $f$ into $f_1$ and $f_2$.

\begin{theorem}
\label{th:decomp}
If $f_1,f_2$ are eigenfunctions with corresponding eigenvalues $\lambda_1,\lambda_2$, with $\lambda_1<\lambda_2$, independent in the sense of Def. \ref{def:independent}, then
$f=f_1+f_2$ can be perfectly decomposed into $f_1$ and $f_2$ using the following spectral decomposition:
$f_1 = LPF_{\frac{1}{\lambda_c}}(f)$, $f_2 = HPF_{\frac{1}{\lambda_c}}(f)$ with $\lambda_1 < \lambda_c < \lambda_2$.
\end{theorem}
\begin{proof}
The theme of the proof is to show that we get an additive spectral response
$$ \phi(t,x) = \delta(t-1/\lambda_1)f_1(x) +  \delta(t-1/\lambda_2)f_2(x)$$
and therefore the spectral filtering proposed above (Eqs. \eqref{eq:phi_ef},\eqref{eq:lpf}, \eqref{eq:hpf} ,which hold for the general one-homogeneous case) decomposes $f$ correctly.

We examine the gradient flow \eqref{eq:grad_flow} with initial conditions $f=f_1+f_2$.
Let us show that given the above assumptions the solution is
\begin{equation}
\label{eq:u_f1f2}
u(t,x) = (1-\lambda_1 t)^+ f_1(x) + (1-\lambda_2 t)^+ f_2(x).
\end{equation}
It is easy to see that for \eqref{eq:u_f1f2} the first time derivative is
$$\partial_t u(t,x) = \left\{
\begin{array}{ll}
-\lambda_1 f_1(x) - \lambda_2 f_2(x),& 0 \le t < 1/\lambda_2\\
-\lambda_1 f_1(x) - \lambda_2 f_2(x),& 1/\lambda_2 \le t < 1/\lambda_1\\
0,& 1/\lambda_1 \le t \\
\end{array}
\right.$$

We now need to check the subdifferential. We do this for $0 \le t < 1/\lambda_2$, similar results can be shown for the other time intervals.
We denote by $p(\cdot)$ an element in $\p J(\cdot)$.

\begin{displaymath}
\begin{array}{lll}
\p J(u(t)) &=& \p J\left((1-\lambda_1 t)f_1(x) + (1-\lambda_2 t)f_2(x)\right) \nonumber \\
&\stackrel{\text{LIS}}{\ni}&  p((1-\lambda_1 t)f_1(x)) + p((1-\lambda_2 t)f_2(x)) \nonumber \\
&\stackrel{\text{Eq.\eqref{eq:p_alpha}}}{=}& p(f_1(x)) + p(f_2(x))  \nonumber \\
&\stackrel{\text{Eq.\eqref{eq:eigenfunction}}}{=}& \lambda_1 f_1(x) + \lambda_2 f_2(x)  \nonumber \\
&=& -\partial_t u(t,x).
\end{array}
\end{displaymath}

\off{
It is easy to show that by the linearity in the subdifferntial we have for $u$ of \eqref{eq:u_f1f2}
\begin{equation}
\label{eq:pu_f1f2}
p(u) = p\left( \alpha_1 f_1 + \alpha_2 f_2\right) = p(\alpha_1 f_1) + p(\alpha_2 f_2),
\end{equation}
where $\alpha_1 = (1-\lambda_1 t)^+$, $\alpha_2 = (1-\lambda_2 t)^+$.
Moreover, for the functional $J(u)$ we have
\begin{equation}
\label{eq:J_f1f2}
\begin{array}{ll}
J(u) & = J( \alpha_1 f_1 + \alpha_2 f_2) \\
&= \< \alpha_1 f_1 + \alpha_2 f_2, p(\alpha_1 f_1 + \alpha_2 f_2)\>,\\
&= \< \alpha_1 f_1 + \alpha_2 f_2, p(\alpha_1 f_1) + p(\alpha_2 f_2)\>\\
&= \< \alpha_1 f_1,p(\alpha_1 f_1)\>  + \< \alpha_2 f_2, p(\alpha_2 f_2)\>.
\end{array}
\end{equation}
SHOULD GO OVER PROOF AND EXPLAIN.
}
\end{proof}

We can conclude that two eigenfunctions with different eigenvalues which are independent, with respect to the regularizer $J$, can be perfectly
decomposed using spectral decomposition based on $J$.


\subsection{Decomposition measures}  
The conditions stated in the above theorem are somewhat strict.
We would like to have a soft measure for the independence of two signals which attains the value $1$ for
completely independent signals (in the sense of Def. \ref{def:independent}) and $0$ for completely correlated signals.
It is expected that this measure will indicate how well two signals can be decomposed.


\subsection{Orthogonality measure}
Let an orthogonality indicator be defined by
\begin{equation}
\label{eq:orth}
\mathcal{O}(u,v) = 1 - \frac{\sqrt{|[u,v]_J[v,u]_J|}}{J(u)J(v)}.
\end{equation}
We have that $0 \le \mathcal{O}(u,v) \le 1$ and $\mathcal{O}=1$ in the orthogonal case, if either $[u,v]_J=0$ or $[v,u]_J=0$.
For the fully correlated case $v=au$, $a >0$, we get $\mathcal{O}(u,au)=0$.
\off{
Let a triangle inequality indicator be defined by
\begin{equation}
\label{eq:tri}
\mathcal{T}(u,v) = \frac{J(u+v)}{J(u)+J(v)}.
\end{equation}
Here also we have $0 \le \mathcal{T}(u,v) \le 1$ and $\mathcal{T}=1$ when the triangle inequality reaches equality.

We show that linearity in the subdif yields that both $\mathcal{O}=1$ and $\mathcal{T}=1$.
Show... (proposition?)

Therefore we suggest the following independence measure $\mathcal{I}(\cdot,\cdot)$:
\begin{equation}
\label{eq:ind}
\mathcal{I}(u,v) = \mathcal{O}(u,v)\cdot\mathcal{T}(u,v).
\end{equation}
Naturally from the bounds on $\mathcal{O}$ and $\mathcal{T}$ we have $0 \le \mathcal{I}(u,v) \le 1$
and for $u,v$ l.i.s. we reach the maximum $\mathcal{I}=1$ (at this point we do not obtain a proof
that this is the only case).
}

\subsection{LIS measure}
Here a more direct relation to the (LIS) property is defined.
We measure how different is $p(u+v)$ from $p(u)+p(v)$. This is done in terms of h.s.i.p.,
\begin{equation}
\label{eq:E}
\begin{array}{lll}
E(u,v) &:=& \lfloor u+v,u \rfloor_J + \lfloor u+v,v \rfloor_J - \lfloor u+v,u+v \rfloor_J\\
& = & \< u+v,p(u)+p(v)-p(u+v)\>.
\end{array}
\end{equation}
We show below that $E \le J(u+v)$. Also we have that $E \to 0$ as $p(u)+p(v) \to p(u+v)$. A possible indicator $L$ for the (LIS) property can therefore be
\begin{equation}
\label{eq:LIS_ind}
L(u,v) := 1-\frac{|E(u,v)|}{J(u+v)}.
\end{equation}
Let us show that
$$E(u,v) \le J(u+v).$$
From \eqref{eq:hsip_ineq} we have
$\lfloor u+v,u \rfloor_J \le \lfloor u+v,u+v \rfloor_J$ and
$\lfloor u+v,v \rfloor_J \le \lfloor u+v,u+v \rfloor_J$,
where $\lfloor u+v,u+v \rfloor_J = J(u+v)$.
Note also that for the fully correlated case, $v=au$, $a >0$, we get $L(u,au)=0$.

\begin{figure}[htb]
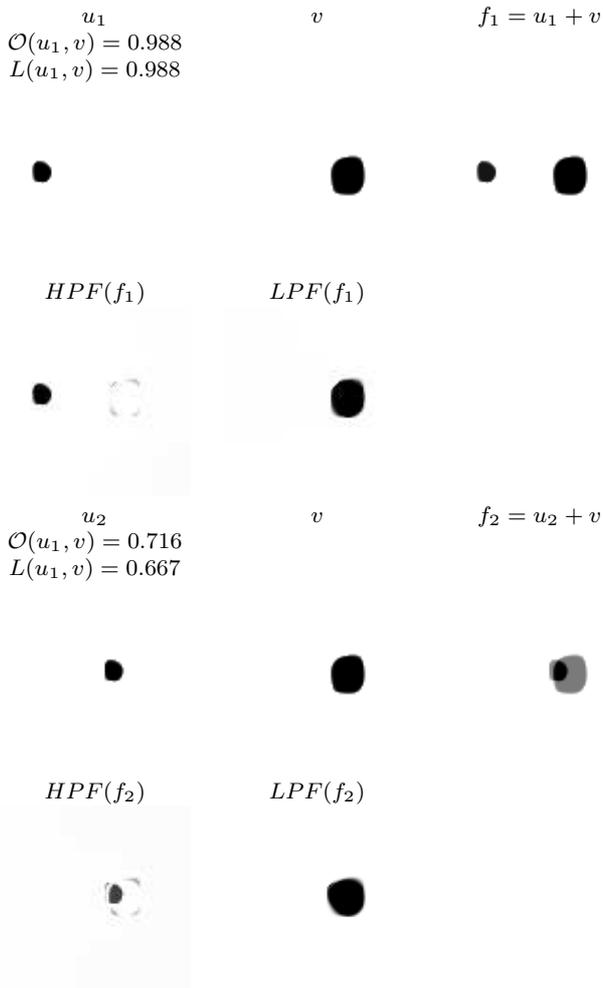

\begin{center}
\begin{tabular}{ ccc }
$u_1$ & $v$ & $f_1=u_1+v$\\
$\mathcal{O}(u_1,v)=0.988$\\
$L(u_1,v)=0.988$\\
\includegraphics[width=25mm, clip]{blobs_u.eps}&
\includegraphics[width=25mm, clip]{blobs_v.eps}&
\includegraphics[width=25mm, clip]{blobs_f.eps}\\
$HPF(f_1)$ & $LPF(f_1)$ \\
\includegraphics[width=25mm, clip]{blobs_f_H.eps}&
\includegraphics[width=25mm, clip]{blobs_res.eps}\\
$u_2$ & $v$ & $f_2=u_2+v$ \\
$\mathcal{O}(u_1,v)=0.716$ \\
$L(u_1,v)=0.667$\\
\includegraphics[width=25mm, clip]{blobs_u2.eps}&
\includegraphics[width=25mm, clip]{blobs_v.eps}&
\includegraphics[width=25mm, clip]{blobs_f2.eps}\\
$HPF(f_2)$ & $LPF(f_2)$ \\
\includegraphics[width=25mm, clip]{blobs_f_H2.eps}&
\includegraphics[width=25mm, clip]{blobs_res2.eps}\\
\end{tabular}
\caption{ Separating blobs of different scale using spectral filtering.
 }
\label{fig:blobs}
\end{center}
\end{figure}

\begin{figure}[htb]
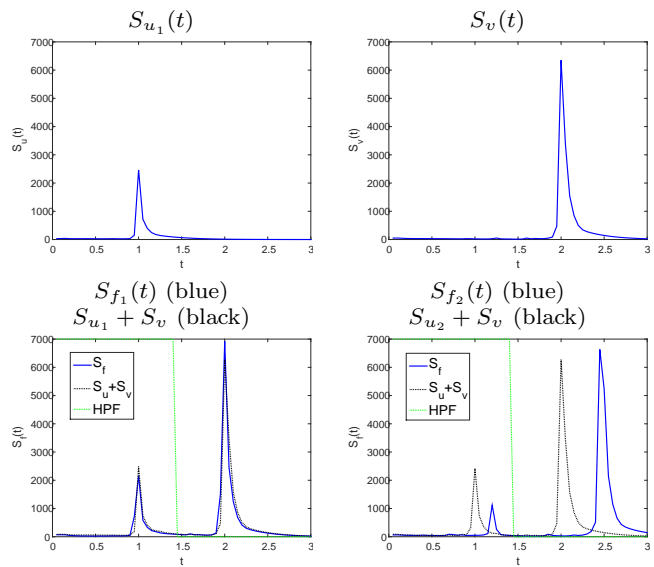

\begin{center}
\begin{tabular}{ cc }
$S_{u_1}(t)$ & $S_v(t)$ \\
\includegraphics[width=40mm, clip]{blobs_Su.eps}&
\includegraphics[width=40mm, clip]{blobs_Sv.eps}\\
$S_{f_1}(t)$ (blue)  & $S_{f_2}(t)$ (blue)\\
$S_{u_1}+S_v$ (black)  & $S_{u_2}+S_v$ (black)\\
\includegraphics[width=40mm, clip]{blobs_Sf1.eps}&
\includegraphics[width=40mm, clip]{blobs_Sf2.eps}\\
\end{tabular}
\caption{ Spectra of the different blob signals.
 }
\label{fig:blobs_S}
\end{center}
\end{figure}

\begin{figure}[htb]
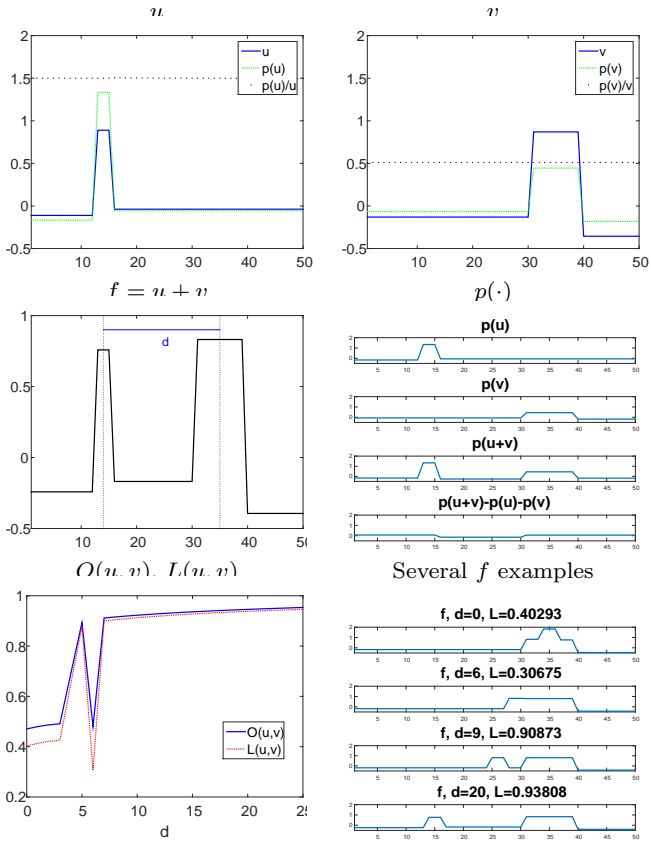

\begin{center}
\begin{tabular}{ cc }
$u$ & $v$\\
\includegraphics[width=40mm]{1d_u_10K_iter_a1.eps}&
\includegraphics[width=40mm]{1d_v_10K_iter_a1.eps}\\
$f=u+v$ & $p(\cdot)$\\
\includegraphics[width=40mm]{1d_f1.eps} &
\includegraphics[width=38mm]{1d_p_a1.eps}\\
$O(u,v)$, $L(u,v)$ & Several $f$ examples \\
\includegraphics[width=40mm]{1d_OL.eps} &
\includegraphics[width=38mm]{1d_fdL.eps}\\
\end{tabular}
\caption{Comparison of $\mathcal{O}(u,v)$ and $L(u,v)$ for the 1D case of $u$ and $v$ being
2 precise TV eigenfunctions in a finite domain. Top row: $u$ (blue), $p(u)$ (green) and $\frac{p(u)}{u}$ (black, dashed) are shown on the left,  $v$, $p(v)$ and $\frac{p(v)}{v}$ (right).
Middle row (from left), $f=u+v$ and $d$ is shown which is the distance between the centers of $u$ and $v$.
An example ($d=21$) of $p(u)$, $p(v)$, $p(u+v)$ and the difference $p(u+v)-p(u)-p(v)$ (from top subplot, respectively). As the difference vanishes the LIS measure $L(u,v)$ approaches 1.
Bottom row, $O(u,v)$ (blue) and $L(u,v)$ (red, dashed) as a function of the distance $d$. On the right,
several cases of $f$ for different values of $d$.
}
\label{fig:1d_sigs}
\end{center}
\end{figure}

\begin{figure}[htb]
\begin{center}
\includegraphics[width=60mm]{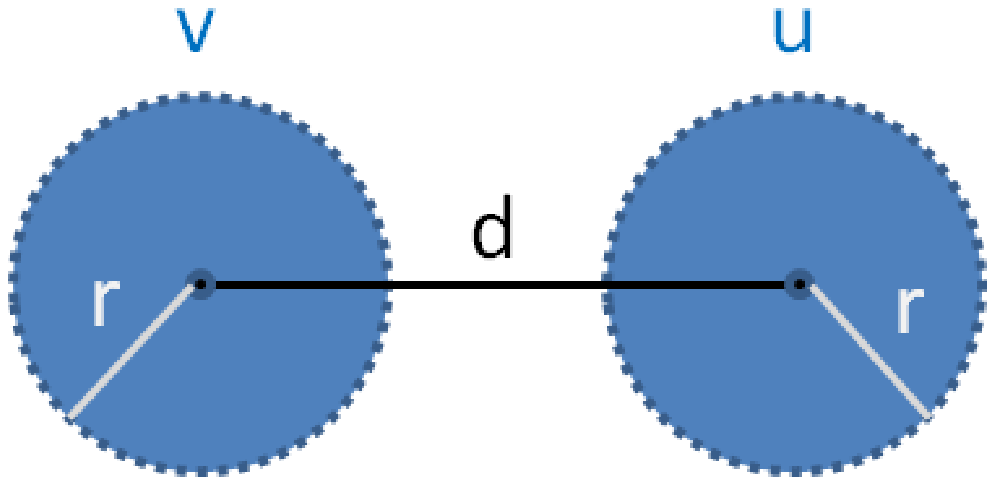}
\includegraphics[width=80mm]{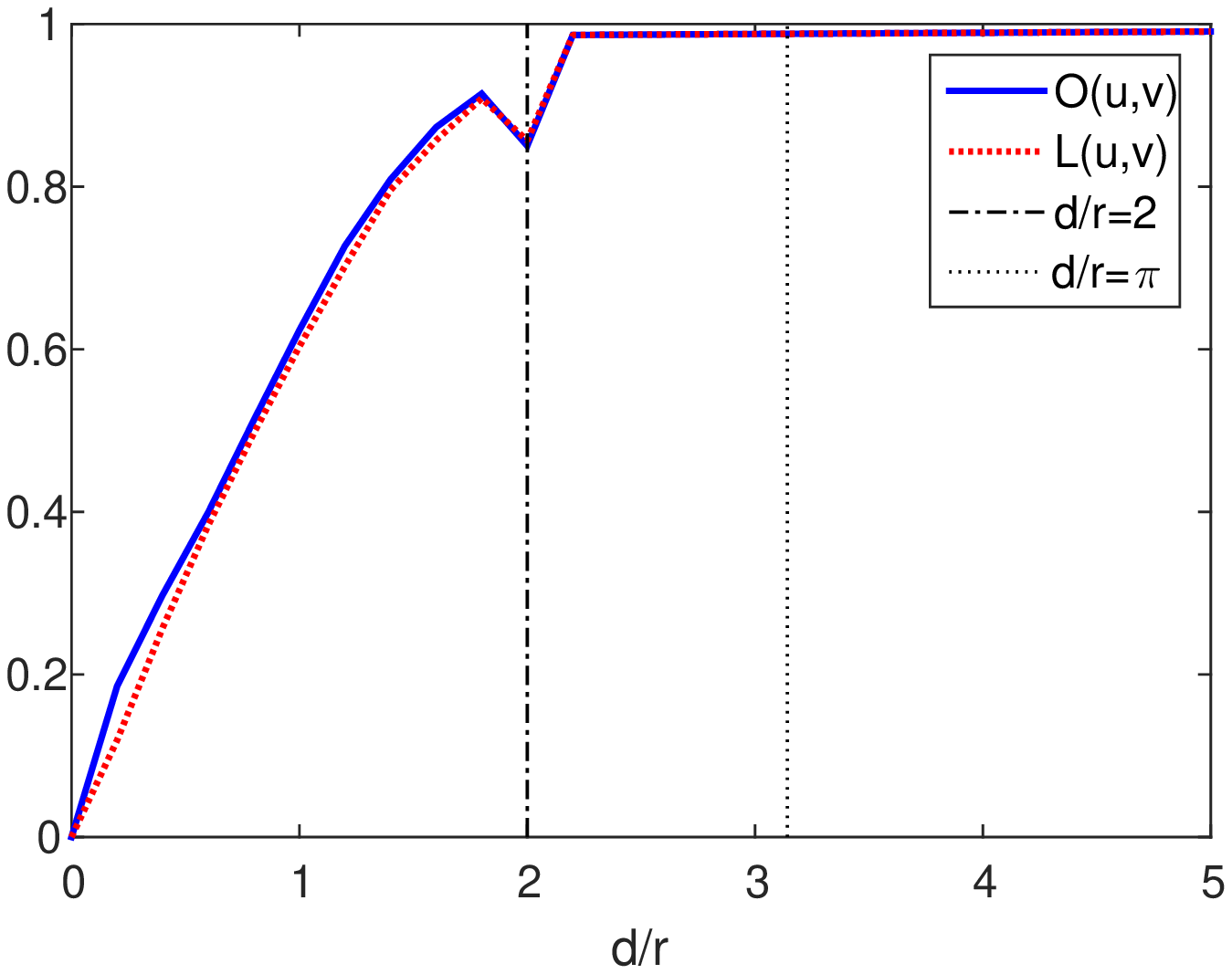}
\caption{Experiment of 2 identical discs. Top -  illustration of the discs $u$ and $v$ of radius $r$ and the distance $d$ between
the centers of the discs. Bottom $\mathcal{O}(u,v)$, Eq. \eqref{eq:orth}, and $L(u,v)$, Eq. \eqref{eq:LIS_ind}, as a function
of $d/r$.}
\label{fig:2circles_diag}
\end{center}
\end{figure}

\begin{figure}[htb]
\begin{center}
\includegraphics[width=90mm]{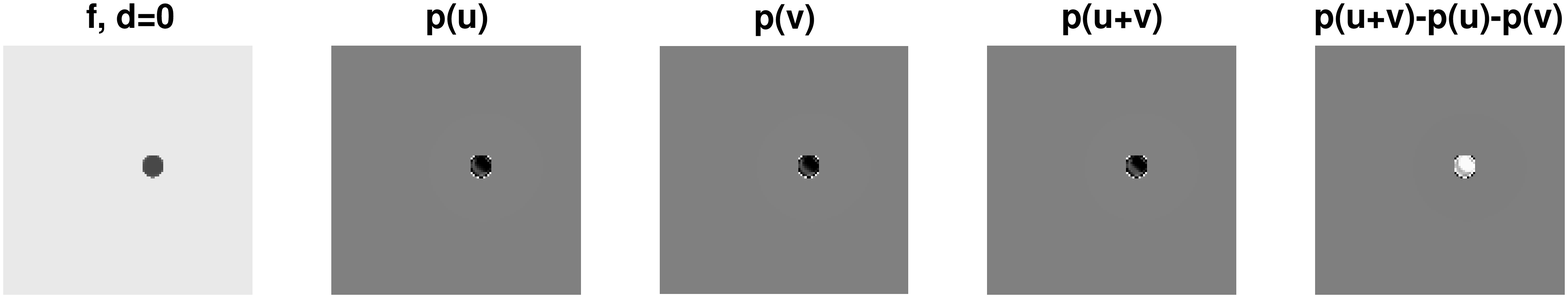}
\includegraphics[width=90mm]{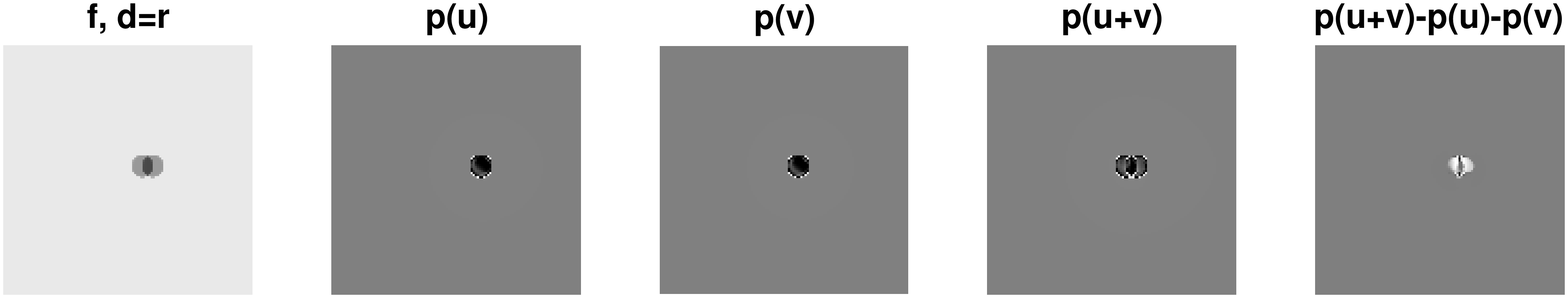}
\includegraphics[width=90mm]{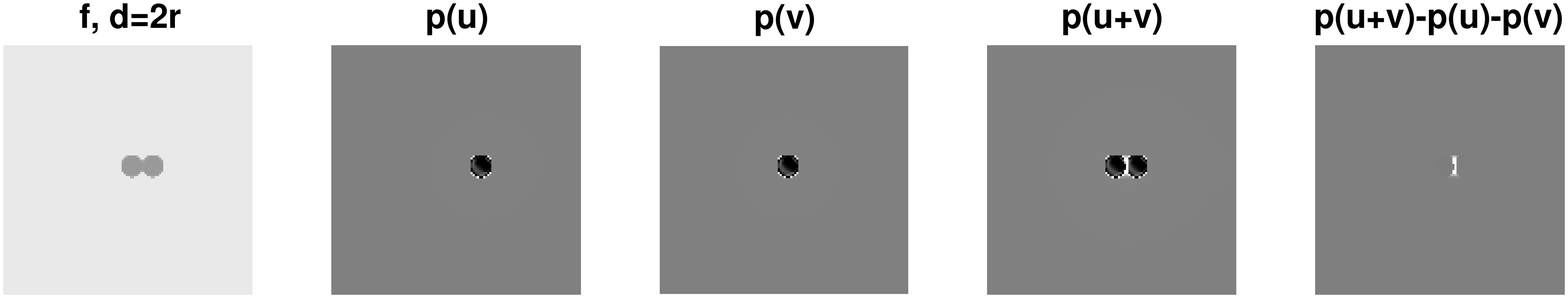}
\includegraphics[width=90mm]{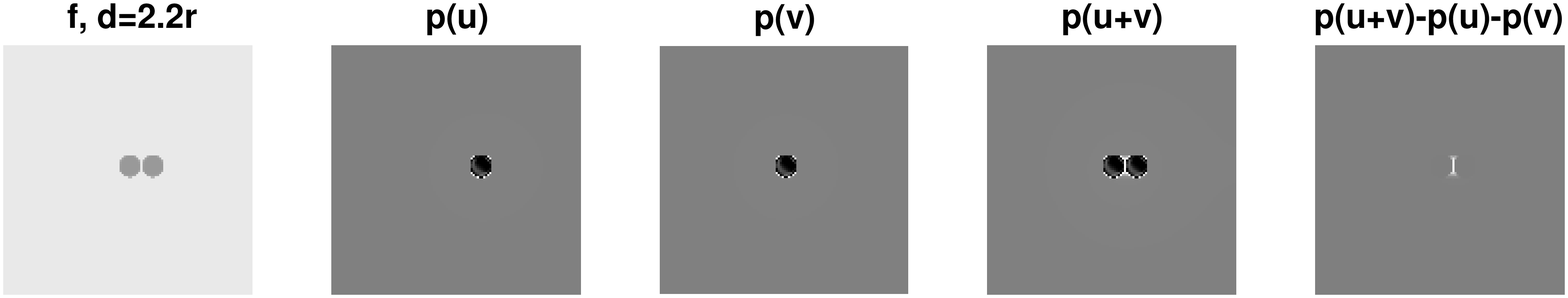}
\includegraphics[width=90mm]{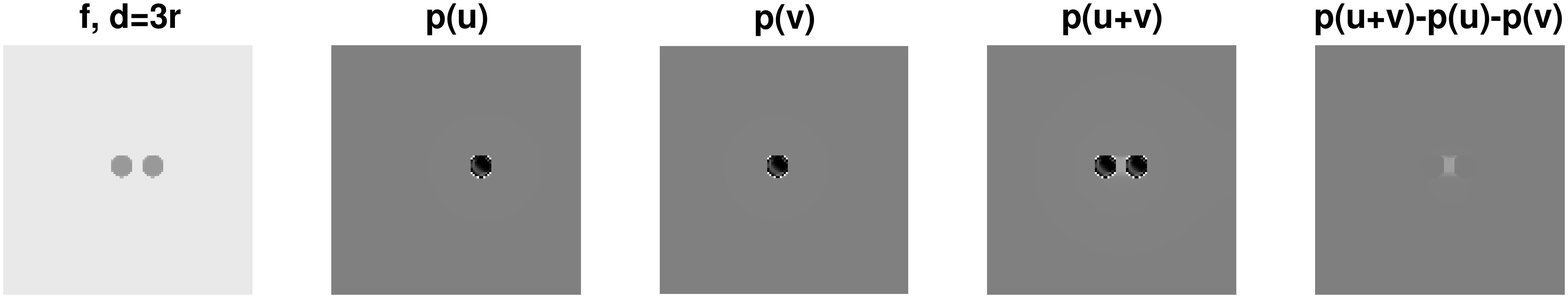}
\includegraphics[width=90mm]{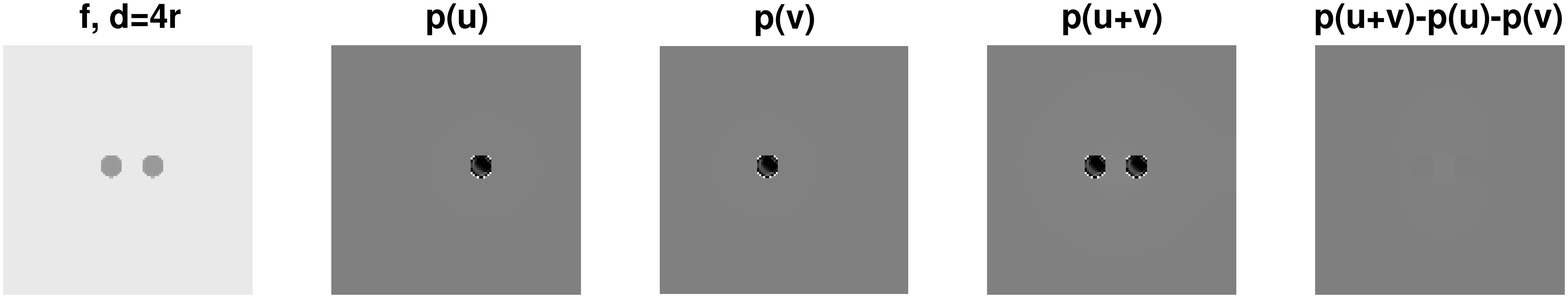}
\caption{A few examples of the 2 discs experiment. Each column (from left):
 $f=u+v$, $p(u)$, $p(v)$, $p(u+v)$, $p(u+v)-p(u)-p(v)$.
 The rows are results of different distances between the discs,
 $d/r=0,1,2,2.2,3,4$.
 }
\label{fig:2circles_r}
\end{center}
\end{figure}

\section{Experiments}
The following experiments are performed to show the behavior of the soft measures described in the previous section for the TV functional.
We compare the orthogonality measure $\mathcal{O}(u,v)$, Eq. \eqref{eq:orth}, and the LIS
measure $L(u,v)$, Eq. \eqref{eq:LIS_ind}.

In Figs. \ref{fig:blobs} and \ref{fig:blobs_S} two cases are shown. In the first one (Fig. \ref{fig:blobs} top 2 rows) $f_1=u_1+v$, where $u_1$ and $v$ are two blobs which are spatially well separated. The decomposition indicators are close to 1 ($\mathcal{O}(u_1,v)=L(u_1,v)=0.988$). A high-pass-filter, as defined in \eqref{eq:hpf}, was used to separate $u_1$ with a cutoff between the peaks, see the green line in Fig. \ref{fig:blobs_S}, bottom left, which visualizes the filter transfer function. One can observe a relatively good separation (with some residual of $v$ as it is not a precise eigenfunction). In Fig. \ref{fig:blobs_S} the spectrum of $u_1$ and $v$ are shown and the spectrum of their sum superimposed on the spectrum of $f_1$ (bottom left), which are close to identical.

The case of overlapping signals $u_2$ and $v$, $f_2=u_2+v$, is shown as well with significantly lower
$\mathcal{O}$ and $L$ indicators and lower quality decomposition (the spectra are also not additive).

In Fig. \ref{fig:1d_sigs} $u$ and $v$ are constructed to be precise discrete eigenfunctions. One can see numerically the pointwise ratio at the top (black, dashed) $\frac{p(u)}{u}$ which is practically constant (for all x). This means that $u$ is indeed an
eigenfunction of TV and admits $p(u)=\lambda u$. The same goes for $v$ on the top right side.
We denote by $d$ the distance between the centers of the peak parts of $u$ and $v$ (shown on
the second row on the left). The eigenfunction $u$ is displaced from being at $d=0$ to $d=25$,
where for each $d$ the measures $\mathcal{O}(u,v)$ and $L(u,v)$ are computed.
Both indicators are well correlated, with $L$ yielding slightly sharper results.
As can be expected, as the peaks of the functions $u$ and $v$ are farther apart, decomposition is easier and both indicator approach 1. Several instances of the composition $f=u+v$ are shown on the bottom
row on the right.

In Figs. \ref{fig:2circles_diag} and \ref{fig:2circles_r} a 2D experiment is shown.
Here $d$ is the distance between the centers of two discs of identical size (radius).
In the continuous case, in an unbounded domain $\mathbb{R}^2$ a disc is an eigenfunction of TV.
Here we have a bounded domain and cannot produce discretely real discs, so this is an approximation.
As those discs are identical (radius and height), in principle they cannot be decomposed through
spectral filtering since they have the same eigenvalue. However we can compare this case to a
theoretical analysis done by Bellettini et al \cite{bellettini2002total}.
It was shown in \cite{bellettini2002total} that for two identical discs of radius $r$ the sum of the two discs is
also an eigenfunction (meaning they admit (LIS)) if $d \ge \pi r$.
Therefore the values of $\mathcal{O}(u,v)$ and $L(u,v)$ are plotted as a function of $\frac{d}{r}$,
with critical points at $\frac{d}{r}=2$, that is the discs are just separated but touch each other at
a single point, and at $\frac{d}{r}=\pi$, the theoretical critical distacne.
As can be seen, $\mathcal{O}$ and $L$ are almost identical here, however the critical point
may not be that significant and as soon as the discs do not touch each other,
  $\frac{d}{r}>2$, the values approach 1 fast.
One can notice in the numerical examples for several $d$ values on the right of Fig. \ref{fig:2circles_r}
that for $\frac{d}{r}=4 > \pi$ indeed we get that $p(u+v)-p(u)-p(v)$ almost vanishes numerically.


\off{
CAN WE SAY $ \| p(f_1+f_2)\| \le \| p(f_1)\| + \| p(f_2)\|$ ?  

For independent signals we have $\| p(f_1+f_2)\| = \|p(f_1)\| + \|p(f_2) \|$ and therefore $\mathcal{I}(f_1,f_2)=1$.
For completely correlated signals $f_2 = \alpha f_1$, $\alpha >0 $, we have
$$\| p(f_1+f_2)\| = \| p((1+\alpha)f_1)\|=\|p(f_1)\|=\|p(f_2)\|,$$ and $\mathcal{I}(f_1,f_2)=0$.
In the case where the signal or part of it are being removed (in the null space) the measure can be negative, and in the
extreme case $f_2 = -f_1$, with $p(0)=0$ we get the complete loss of both signals and
$\mathcal{I}(f_1,f_2)=-1$.
} 

\section{Conclusion}
In this work several new concepts were presented, which can be helpful in future theoretical understanding and better
employment of  convex regularizers.
The properties of semi-inner-products for convex functionals were stated, following the s.i.p. of Lumer for normed spaces.
Essentially, linearity and homogeneity are kept in the first argument, the functional is induced by the s.i.p. and a Cauchy-Schwartz-type
property holds. The s.i.p., however, does not behave linearly with respect to the second argument.
For non-smooth functionals s.i.p.'s are similar to the subdifferential and may contain several elements (in this case
it is unique when a subgradient element is chosen).

For the one-homogeneous case a general formulation of the s.i.p. was given. This yields natural definitions of orthogonality and angles
between 2 functions, with respect to regularizing functionals like TV or TGV.
The relation to the Bregman distance was shown, where in the one-homogeneous case the Bregman distance between two functions can be expressed in terms of
the angle between those functions.

An extension of s.i.p.'s to general degrees was suggested, where the case of half-semi-inner-products (h.s.i.p.) was further developed.
Finally, it was shown that when the h.s.i.p. is linear in the second argument one can decompose two eigenfunctions (with different eigenvalues) perfectly, using the spectral filters proposed in \cite{Gilboa_spectv_SIAM_2014,spec_one_homog15}.
As the conditions for perfect decomposition are quite strict, two soft indicators based on s.i.p.'s and h.s.i.p.'s were suggested. Their goal is to measure how close we are to fulfilling those conditions. Initial experiments indicate both measures are useful in assessing the separability of signals with a dominant scale (where the one based on the (LIS) property yields slightly sharper results).

\small{
\bibliography{guy_bib}

\begin{thebibliography}{10}

\bibitem{tvFlowAndrea2001}
F.~Andreu, C.~Ballester, V.~Caselles, and J.~M. Maz{\'o}n.
\newblock Minimizing total variation flow.
\newblock {\em Differential and Integral Equations}, 14(3):321--360, 2001.

\bibitem{tv_flow}
F.~Andreu, C.~Ballester, V.~Caselles, and J.~M. Maz{\'o}n.
\newblock Minimizing total variation flow.
\newblock {\em Differential and Integral Equations}, 14(3):321--360, 2001.

\bibitem{andreu2002some}
F.~Andreu, V.~Caselles, JI~D{\i}az, and JM~Maz{\'o}n.
\newblock Some qualitative properties for the total variation flow.
\newblock {\em Journal of Functional Analysis}, 188(2):516--547, 2002.

\bibitem{ak_book02}
G.~Aubert and P.~Kornprobst.
\newblock {\em Mathematical Problems in Image Processing}, volume 147 of {\em
  Applied Mathematical Sciences}.
\newblock Springer-Verlag, 2002.

\bibitem{Aujol[3]}
J.F. Aujol, G.~Aubert, L.~Blanc-F{\'e}raud, and A.~Chambolle.
\newblock Image decomposition into a bounded variation component and an
  oscillating component.
\newblock {\em JMIV}, 22(1), January 2005.

\bibitem{agco06}
J.F. Aujol, G.~Gilboa, T.~Chan, and S.~Osher.
\newblock Structure-texture image decomposition -- modeling, algorithms, and
  parameter selection.
\newblock {\em International Journal of Computer Vision}, 67(1):111--136, 2006.

\bibitem{Clustering_Bregman_div_JMRL2005}
A.~Banerjee, S.~Merugu, I.~Dhillon, and J.~Ghosh.
\newblock Clustering with bregman divergences.
\newblock {\em The Journal of Machine Learning Research}, 6:1705--1749, 2005.

\bibitem{bellettini2002total}
G.~Bellettini, V.~Caselles, and M.~Novaga.
\newblock The total variation flow in {$R^N$}.
\newblock {\em Journal of Differential Equations}, 184(2):475--525, 2002.

\bibitem{Benning_Burger_2013}
M.~Benning and M.~Burger.
\newblock Ground states and singular vectors of convex variational
  regularization methods.
\newblock {\em Methods and Applications of Analysis}, 20(4):295--334, 2013.

\bibitem{bredies_tgv_2010}
K.~Bredies, K.~Kunisch, and T.~Pock.
\newblock Total generalized variation.
\newblock {\em SIAM J. Imaging Sciences}, 3(3):492--526, 2010.

\bibitem{bre67}
L.M. Bregman.
\newblock The relaxation method for finding the common point of convex sets and
  its application to the solution of problems in convex programming.
\newblock {\em USSR Comp. Math. and Math. Phys.}, 7:200--217, 1967.

\bibitem{Bruck1973}
R.~Bruck.
\newblock Nonexpansive projections on subsets of banach spaces.
\newblock {\em Pacific Journal of Mathematics}, 47(2):341--355, 1973.

\bibitem{Burger_Bregman_IP_2015}
M.~Burger.
\newblock Bregman distances in inverse problems and partial differential
  equation.
\newblock {\em arXiv preprint arXiv:1505.05191}, 2015.

\bibitem{spec_one_homog15}
M.~Burger, L.~Eckardt, G.~Gilboa, and M.~Moeller.
\newblock Spectral representations of one-homogeneous functionals.
\newblock In {\em Scale Space and Variational Methods in Computer Vision},
  pages 16--27. Springer, 2015.

\bibitem{iss}
M.~Burger, G.~Gilboa, S.~Osher, and J.~Xu.
\newblock Nonlinear inverse scale space methods.
\newblock {\em Comm. in Math. Sci.}, 4(1):179--212, 2006.

\bibitem{TV_Zoo_Burger_Osher2013}
M.~Burger and S.~Osher.
\newblock A guide to the tv zoo.
\newblock In {\em Level Set and PDE Based Reconstruction Methods in Imaging},
  pages 1--70. Springer, 2013.

\bibitem{Censor_multiprojection_Bregman_1994}
Y.~Censor and T.~Elfving.
\newblock A multiprojection algorithm using bregman projections in a product
  space.
\newblock {\em Numerical Algorithms}, 8(2):221--239, 1994.

\bibitem{Intro_TV_Chambolle2010}
A.~Chambolle, V.~Caselles, D.~Cremers, M.~Novaga, and T.~Pock.
\newblock An introduction to total variation for image analysis.
\newblock {\em Theoretical foundations and numerical methods for sparse
  recovery}, 9:263--340, 2010.

\bibitem{Nemirovski_subgrad_algo_2014}
B.~Cox, A.~Juditsky, and A.~Nemirovski.
\newblock Dual subgradient algorithms for large-scale nonsmooth learning
  problems.
\newblock {\em Mathematical Programming}, 148(1-2):143--180, 2014.

\bibitem{classification_Banach_Der_Lee_2007}
R.~Der and D.~Lee.
\newblock Large-margin classification in banach spaces.
\newblock In {\em International Conference on Artificial Intelligence and
  Statistics}, pages 91--98, 2007.

\bibitem{sip_book2004}
S.~S. Dragomir.
\newblock {\em Semi-inner products and applications}.
\newblock Nova Science Publishers New York, 2004.

\bibitem{Ekeland_Temam1999}
I~Ekeland and R~T{\'e}mam.
\newblock Convex analysis and variational problems.
\newblock {\em Classics in Applied Mathematics. Society for Industrial and
  Applied Mathematics}, 1999.

\bibitem{tvf_giga2010}
Y.~Giga and R.V. Kohn.
\newblock Scale-invariant extinction time estimates for some singular diffusion
  equations.
\newblock {\em Hokkaido University Preprint Series in Mathematics}, (963),
  2010.

\bibitem{Gilboa_spectv_SIAM_2014}
G.~Gilboa.
\newblock A total variation spectral framework for scale and texture analysis.
\newblock {\em SIAM J. Imaging Sciences}, 7(4):1937--1961, 2014.

\bibitem{Gilboa_ef_flow}
G.~Gilboa.
\newblock Flows generating nonlinear eigenfunctions, 2015.
\newblock In preparation.

\bibitem{jmivPreprint}
G.~Gilboa, M.~Moeller, and M.~Burger.
\newblock Nonlinear spectral analysis via one-homogeneous functionals -
  overview and future prospects.
\newblock Submitted. Preprint at {http://arxiv.org/abs/1510.01077}.

\bibitem{gsz_var06}
G.~Gilboa, N.~Sochen, and Y.Y. Zeevi.
\newblock Variational denoising of partly-textured images by spatially varying
  constraints.
\newblock {\em IEEE Trans. on Image Processing}, 15(8):2280--2289, 2006.

\bibitem{Giles67}
J.R. Giles.
\newblock Classes of semi-inner-product spaces.
\newblock {\em Transactions of the American Mathematical Society}, pages
  436--446, 1967.

\bibitem{SplitBregman2009}
T.~Goldstein and S.~Osher.
\newblock The split bregman method for l1-regularized problems.
\newblock {\em SIAM Journal on Imaging Sciences}, 2(2):323--343, 2009.

\bibitem{Metric_learning_JMRL2012}
P.~Jain, B.~Kulis, J.V. Davis, and I.S. Dhillon.
\newblock Metric and kernel learning using a linear transformation.
\newblock {\em The Journal of Machine Learning Research}, 13(1):519--547, 2012.

\bibitem{NL_PCA2015}
M.~Lange, M.~Biehl, and T.~Villmann.
\newblock Non-euclidean principal component analysis by hebbian learning.
\newblock {\em Neurocomputing}, 147:107--119, 2015.

\bibitem{Lumer61}
G.~Lumer.
\newblock Semi-inner-product spaces.
\newblock {\em Transactions of the American Mathematical Society}, pages
  29--43, 1961.

\bibitem{Ma_Goldfarb_rank_min_2011}
S.~Ma, D.~Goldfarb, and L.~Chen.
\newblock Fixed point and bregman iterative methods for matrix rank
  minimization.
\newblock {\em Mathematical Programming}, 128(1-2):321--353, 2011.

\bibitem{Meyer[1]}
Y.~Meyer.
\newblock Oscillating patterns in image processing and in some nonlinear
  evolution equations, March 2001.
\newblock The 15th Dean Jacquelines B. Lewis Memorial Lectures.

\bibitem{Bregman_obgxy}
S.~Osher, M.~Burger, D.~Goldfarb, J.~Xu, and W.~Yin.
\newblock An iterative regularization method for total variation based image
  restoration.
\newblock {\em SIAM Journal on Multiscale Modeling and Simulation}, 4:460--489,
  2005.

\bibitem{Luminita[2]}
S.~Osher, A.~Sole, and L.~Vese.
\newblock Image decomposition and restoration using total variation
  minimization and the {H}$^{-1}$ norm.
\newblock {\em SIAM Multiscale Modeling and Simulation}, 1(3):349--370, 2003.

\bibitem{rof92}
L.~Rudin, S.~Osher, and E.~Fatemi.
\newblock Nonlinear total variation based noise removal algorithms.
\newblock {\em Physica D}, 60:259--268, 1992.

\bibitem{Luminita[1]}
L.~Vese and S.~Osher.
\newblock Modeling textures with total variation minimization and oscillating
  patterns in image processing.
\newblock {\em Journal of Scientific Computing}, 19:553--572, 2003.

\bibitem{Tai_multilayer_graph_cut_seg_PR2013}
Y.~Yang, S.~Han, T.~Wang, W.~Tao, and X.-C. Tai.
\newblock Multilayer graph cuts based unsupervised color--texture image
  segmentation using multivariate mixed student's t-distribution and regional
  credibility merging.
\newblock {\em Pattern Recognition}, 46(4):1101--1124, 2013.

\bibitem{zhang2009reproducing}
H.~Zhang, Y.~Xu, and J.~Zhang.
\newblock Reproducing kernel banach spaces for machine learning.
\newblock {\em The Journal of Machine Learning Research}, 10:2741--2775, 2009.

\bibitem{Zhang_Burger_Bresson_Osher_NL_Bregman_2010}
X.~Zhang, M.~Burger, X.~Bresson, and S.~Osher.
\newblock Bregmanized nonlocal regularization for deconvolution and sparse
  reconstruction.
\newblock {\em SIAM Journal on Imaging Sciences}, 3(3):253--276, 2010.

\end{thebibliography}

}
\end{document}